\documentclass[12pt,article]{amsart}

\usepackage[margin=1in]{geometry}
\usepackage{amsmath,amssymb,amsthm,url,mathrsfs,graphicx,amscd,amsfonts}
\usepackage[mathscr]{eucal}
\usepackage{dsfont}
\usepackage{mathtools}
\usepackage{hyperref}
\usepackage{bbm}
\usepackage{todonotes}
\usepackage{bm}
\allowdisplaybreaks
\newtheorem{theorem}{Theorem}[section]
\newtheorem{lemma}[theorem]{Lemma}
\newtheorem{corollary}[theorem]{Corollary}

\newtheorem*{conjecture}{Conjecture}
\newtheorem*{theorem*}{Theorem}
\theoremstyle{remark}\newtheorem{remark}{Remark}

\newtheorem{case}{Case}

\numberwithin{equation}{section}

\renewcommand{\pmod}[1]{\left(\mathrm{mod}\,#1\right)}

\newcommand{\ve}{\varepsilon}
\newcommand{\eps}{\epsilon}

\newcommand\blfootnote[1]{
\begingroup
\renewcommand\thefootnote{}\footnote{#1}
\endgroup
}

\begin{document}

\title[Correlation of $L$-functions over function fields]{Correlation of shifted values of $L$-functions in the hyperelliptic ensemble}

\author{Pranendu Darbar}
\address{Indian Statistical Institute   \\
Kolkata  \\
West Bengal 700108,  India}
\email[Pranendu Darbar]{darbarpranendu100@gmail.com}

\author{Gopal Maiti}
\address{Indian Statistical Institute   \\
Kolkata  \\
West Bengal 700108,  India}
\email[Gopal Maiti]{g.gopaltamluk@gmail.com}

%\author{\textcolor{red}{Advika}}


 \begin{abstract}
  	The moments of quadratic Dirichlet $L$-functions over function fields  have recently attracted much attention with the work of Andrade and Keating. In this article, we establish lower bounds for the mean values of the product of  quadratic Dirichlet $L$-functions associated with hyperelliptic curves of genus $g$ over a fixed finite field $\mathbb{F}_q$ in the large genus limit. By using the idea of A. Florea \cite{FL3},  we also obtain their upper bounds. As a consequence, we find upper bounds of its derivatives.  
These lower and upper bounds give the correlation of quadratic Dirichlet $L$-functions associated with hyperelliptic curves with different transitions.
  \end{abstract}

%  File \jobname.tex.\par

\maketitle

 \section{Introduction} 
 The correlation of $L$-functions i.e., study of the mean values of product of shifted values of $L$-functions near the critical line has become central to number theory. 
Random matrix theory has recently become a fundamental tool for understanding the correlation of $L$-functions. Montgomery \cite{MON} showed that two-point correlations between the
non-trivial zeros of the Riemann $\zeta$-function, on the scale of the mean zero spacing, are similar to the corresponding correlations between the eigenvalues of random unitary matrices
in the limit of large matrix size and conjectured that these correlations are, in fact, identical to each other.

\vspace{1mm}
\noindent
Keating and Snaith \cite{KS2} suggested that the value distribution of the Riemann zeta function on its critical line is related to that the characteristic polynomials of random unitary matrices. Conjectures for the moments of $L$-functions have been attempted for many decades, with very little progress until the random matrix theory came into the subject.

\vspace{1mm}
\noindent
The main observation is that the structure of the mean values of
L-functions is more clearly revealed if one considers the average of a product of $L$-functions,
where each $L$-function is evaluated at a location slightly shifted from the critical point. 
%The example mean values given in the previous section can be obtained by allowing the shifts to
%tend to zero.

\vspace{1mm}
\noindent
In this article, we discus about the moments and correlation of Riemann zeta function (belonging to unitary family), Dirichlet $L$-functions (contained in the symplectic family) and quadratic Dirichlet $L$-functions associated with hyperelliptic curves of large genus over a fixed finite field which are also members of the symplectic family. These families and their random matrix analogs have been discussed from the perspective of the leading terms in the asymptotic expressions by several authors (see \cite{CF}, \cite{CFKRS}, \cite{CFKRS2}, \cite{AK2}, \cite{KS1}, \cite{KS2}, \cite{CS} and \cite{KO}).    

\vspace{1mm}

Our main goal of this article is to establish lower and upper bounds for the correlation of shifted values of quadratic Dirichlet $L$-functions near the critical line associated to  the hyperelliptic curves of large genus over a fixed finite field.
 \subsection{Moments of the Riemann Zeta function} A classical question in the theory of Riemann zeta function is to determine the asymptotic behaviour of 
 \[
 M_k(T):= \int_{1}^{T}|\zeta(\tfrac{1}{2}+it)|^{2k}\,dt,
 \]
  where $k\in \mathbb{C}$, as $T\to \infty$. It is believed that for a given positive real number $k$, 
  \[
  \displaystyle{M_{k}(T)\sim c_{k}T(\log T)^{k^{2}}},
  \]
 where $c_k$ is a positive constant. Ramachandra \cite{RAM} showed that 
 \[
 M_k(T) \gg T(\log T)^{k^{2}},
 \]
 for any $k\in \mathbb{N}$.
  Using moments of characteristic polynomials of random matrices, Keating and Snaith \cite{KS1}, conjectured an exact value of $c_{k}$ for $\Re (k)>-\frac{1}{2}$. 
 Assuming the Riemann hypothesis (RH), Soundararajan \cite{SOU} showed that for every positive real number $k$ and $\ve>0$,
  \begin{align}\label{L1}
 M_{k}(T)\ll_{k,\ve} c_{k}T(\log T)^{k^{2}+\ve}
 \end{align} 
 and Harper \cite{AH} removed the exponent $\epsilon$ in the bound of  \eqref{L1}. 
 
 \vspace{2mm}
 \noindent
 A generalization of the moments of $\zeta(s)$ are the shifted moments, defined as $$M_{\bm{k}^{(m)}}(T,\bm{\alpha}^{(m)}):=\int_{0}^{T}|\zeta(\tfrac{1}{2}+it+i\alpha_{1})|^{2k_{1}}\ldots |\zeta(\tfrac{1}{2}+it+i\alpha_{m})|^{2k_{m}}\,dt, $$ where $\bm{k}^{(m)}=(k_{1},\ldots, k_{m})$ is a sequence of positive real numbers and $\bm{\alpha}^{(m)}=(\alpha_{1},\ldots,\alpha_{m})\in \mathbb{R}^{m}$ with $\alpha_{i}\neq\alpha_{j}$ for $i\neq j$. In \cite{CHA}, Chandee obtained lower and upper bounds of $M_{\bm{k}^{(m)}}(T,\bm{\alpha}^{(m)})$ for some special choices of $\bm{\alpha}^{(m)}$ and for a large values of $T$. More precisely, she proved assuming the RH, 
 \[
 M_{\bm{k}^{(m)}}(T,\bm{\alpha}^{(m)})\ll_{\bm{k}^{(m)}, \ve} T\,(\log T)^{k_{1}^{2}+\ldots+k_{m}^{2} +\ve}  \prod_{i<j}\left(\min\Big\{\tfrac{1}{|\alpha_{i}-\alpha_{j}|},\log T\Big\} \right)^{2k_{i}k_{j}} 
 \]
 and unconditionally  
 \begin{align*}
  M_{\bm{k}^{(m)}}(T,\bm{\alpha}^{(m)})\gg_{\bm{k}^{(m)}, \bm{\alpha}^{(m)}} T\,(\log T)^{k_{1}^{2}+\ldots+k_{m}^{2}}  \prod_{i<j}\left(\min\Big\{\frac{1}{|\alpha_{i}-\alpha_{j}|},\log T\Big\} \right)^{2k_{i}k_{j}},
 \end{align*} for sufficiently large $T$.
 \vspace{1mm}
 
 \noindent
  The moments of the derivatives of the Riemann zeta function were studied by several mathematicians. An analog of Soundararajan's estimate \eqref{L1} for the derivatives of Riemann zeta function was obtained by Milinovich \cite{MI2}. Under the RH, he showed that for every $\ve>0$, 
 \[
 \int_{1}^{T}|\zeta^{(l)}(\tfrac{1}{2}+it)|^{2k}\,dt\ll_{k,l,\ve} T\,(\log T)^{k^{2}+ 2kl+\ve},
 \]
  where $k,l\in \mathbb{N}$ and $\zeta^{(l)}$ is the $l$-th derivative of $\zeta$.
 \subsection{Moments of quadratic Dirichlet $L$-functions} Let $\chi_{d}$ be a real primitive Dirichlet character modulo $d$ given by the Kronecker symbol $\chi_{d}(n)=\left(\tfrac{d}{n} \right)$. It is interesting to determine the asymptotic behaviour of $\textstyle{\sum_{0<d\leq D} L(\tfrac{1}{2},\chi_{d})^{k}}\,\,\,$ as $D\rightarrow \infty$. Extending their approach to the zeta function, using random matrix theory, Keating and Snaith \cite{KS2} made the following conjecture about the asymptotic behaviour of moments of Dirichlet $L$-functions $L(\frac{1}{2},\chi_{d})$.
 \begin{conjecture}[Keating, Snaith] For $k$ fixed with $\Re(k)\geq 0$, as $D\rightarrow \infty$ 
 	\begin{align*}
 	\frac{1}{D}\sideset{}{^\star}\sum_{|d|\leq D}L(\tfrac{1}{2}, \chi_{d})^{k}\sim c_{k}\,\left(\log D \right)^{\frac{k(k+1)}{2}} ,
 	\end{align*}
 	where $c_{k}$ is a positive constant and $\normalsize\sideset{}{^\star}\sum$ indicates that the sum is over fundamental discriminants.
 \end{conjecture} 

 \noindent
 In \cite{RS}, Rudnick and Soundararajan obtained that for any rational number $k\geq 1$,
 \begin{align*}
 \frac{1}{D}\sideset{}{^\star}\sum_{|d|\leq D}L(\tfrac{1}{2}, \chi_{d})^{k}\gg_{k} \,\left(\log D \right)^{\frac{k(k+1)}{2}}.
\end{align*} 
 Assuming the generalized Riemann hypothesis (GRH), Soundararajan established that for any positive real number $k$ and $\epsilon>0$,
 \begin{align*}
\frac{1}{D} \sideset{}{^\star}\sum_{|d|\leq D}L(\tfrac{1}{2}, \chi_{d})^{k}\ll_{k, \epsilon}\,\left(\log D \right)^{\frac{k(k+1)}{2}+\ve}.
\end{align*}
\noindent
In general, it is important to find asymptotic behaviour of the following correlation of shifted values of Dirichlet $L$-functions:
\begin{align*}
S_{\bm{k}^{(m)}}(\bm{\alpha}^{(m)}, D):=\sideset{}{^\star}\sum_{d\leq D}L(\tfrac{1}{2}+\alpha_{1}, \chi_{d})^{k_{1}}\dots L(\tfrac{1}{2}+\alpha_{m}, \chi_{d})^{k_{m}},
\end{align*}
where $\bm{k}^{(m)}=(k_{1},\dots,k_{m})$ be a sequence of real numbers and $\bm{\alpha}^{(m)}=(\alpha_{1},\dots,\alpha_{m})$ be a sequence of complex numbers with $\alpha_i \neq\alpha_j$ for $i\neq j$.

Conrey $et.\, al$. \cite{CFKRS} gave a conjecture on the asymptotic behaviour of $S_{\bm{k}^{(m)}}(\bm{\alpha}^{(m)}, D)$.
Analogous questions for higher degree $L$-functions have been studied by Milinovich and Turnage-Butterbaugh \cite{MBT}.
%In particular, assuming the GRH, Milinovich and Turnage-Butterbauge \cite{MU} proved that for every $\epsilon>0$,
%%S(\bm{0}^{(m)}, D)\ll_{\epsilon} D\,\left(\log D\right)^{\frac{k_{1}(k_{1}+1)}{2}+\ldots+\frac{k_{m}(k_{m}+1)}{2}+\ve},
%\end{align*} 
%where $\bm{0}^{(m)}=(0,\ldots, 0)$.
\subsection{Moments of $L$-functions in the hyperelliptic essemble}
%%and quadratic Dirichlet $L$-functions, in this section, we shall discuss about the moments of the quadratic Dirichlet $L$-functions in the hyperelliptic ensemble.

%\todo[inline]{We discussed these basics in the section 2 so that we can ignore it.}
\vspace{2mm} 
\noindent
Let $\mathbb{F}_{q}$ be a finite field of odd cardinality and $\mathbb{F}_{q}[t]$ be the polynomial ring over $\mathbb{F}_{q}$ in variable $t$. Let $D\in \mathbb{F}_q[t]$ be a monic square-free polynomial. The quadratic character $\chi_D$ attached  to $D$ is defined using   quadratic residue symbol for $\mathbb{F}_{q}[t]$ by $\chi_{D}(f)=\left(\frac{D}{f}\right)$ and the corresponding Dirichlet $L$-function is denoted by $L(s, \chi_D)$. It is often convenient to work with the equivalent $L$-function $\mathcal{L}(u, \chi_D)$ written in terms of the variable $u=q^{-s}$.
 
 %For such character their corresponds to $\chi_{D}$ on $L$-function  
% \begin{align*}
%L(s,\chi_{D})=\sum_{f\,\text{monic}} \frac{\chi_{D}(f)}{|f|^{s}}=\prod_{P\,\text{prime}}\left(1-\chi_{D}(P)\,|P|^{-s} \right)^{-1}, \quad \Re(s)>1.
 %\end{align*}
%%%\end{align*}

\vspace{1mm}
\noindent
Define hyperelltiptic ensemble $\mathcal{H}_{n,q}$ or simply $\mathcal{H}_{n}$ as
\[
\mathcal{H}_n=\left\{ D\in \mathbb{F}_{q}[t] : \, D \text{ is monic, square free, and} \, \deg(D)=n \right\}.
\] 
For each $D$ in the Hyperelliptic essemble $\mathcal{H}_n$, there is an associated hyperelliptic curve given by $C_D \,:\, y^2=D(t)$. These curve are non-singular and of genus $g$ given by 
\begin{align}\label{genus}
2g=n-1-\lambda,
\end{align}
where
\begin{align*}
  \lambda= 
 \left\{
 \begin{array}
 [c]{ll}
  1, &\, \text{if\, $n$ even}, \\
 0, &\, \text{if\, $n$ odd}.
  \end{array}
 \right.
\end{align*}
Note that, $g\to \infty$ as $n$ does so.
See section \ref{prilimimaries} for more details about the properties of Dirichlet $L$-function $L(s, \chi_D)$ and their spectral interpretation.
%The $L$-function $L(s,\chi_D)$ associated to $\chi_D$ is indeed the numerator of the zeta function associated to the Hyperelliptic curve $C_D$ ( see ...... ). The $L(s,\chi_{D})$ is a polynomial in $u=q^{-s}$ of degree $2g$.

%\begin{remark}
%From \eqref{genus}, 
%\end{remark}
% One introduces the parameter $g$ (known as genus) as it helps to give a uniform statement about the number of non-trivial zeros associated to $L(s,\chi_D)$ regardless of whether $n=2g+1$ or $n=2g+2$.  
%
%\vspace{2mm}
% Recently the study of the moments of quadratic Dirichlet $L$-functions becomes central topics  in the function fields, i.e, for any $k\in \mathbb{N}$
%\begin{align*}
%\sum_{D\in \mathcal{H}_{n}}L(\tfrac{1}{2},\chi_{D})^{k}, \quad %\text{ as } n\to \infty.
%\end{align*} 
\vspace{1mm}
\noindent
Andrade and Keating \cite{AK2} conjectured that as $g\to \infty$,
\begin{align}\label{conjecture on function field}
\sum_{D\in \mathcal{H}_{2g+1}}L(\tfrac{1}{2},\chi_{D})^{k}\, =\,q^{2g+1}\,\left(P_{k}(2g+1) \,+\, o(1) \right),
\end{align}
where $P_{k}$ is a polynomial of degree $\frac{k(k+1)}{2}$. Assuming $q\equiv 1 \pmod 4$, the conjecture \eqref{conjecture on function field} is known for $k=1$ from the work of Andrade \cite{AK1}, and the error term in the asymptotic formula was improved by Florea \cite{FL1}. In \cite{FL2, FL3}, Florea also proved the conjecture \eqref{conjecture on function field} for $k=2, 3$ and $4$ assuming $q\equiv 1 \pmod 4$. For $n=2g+2$, Jung \cite{JUNG} obtained that
\[
\frac{1}{|\mathcal{H}_{2g+2}|}\sum_{D\in \mathcal{H}_{2g+2}}L(\tfrac{1}{2},\chi_{D})=P(1)(g+1)+\frac{P'(1)}{\log q}-P(1)\zeta_{\mathbb{A}}\left(\frac{1}{2}\right)+O\left(2^{g+1}q^{-\frac{g}{2}}\right),
\]
where $P(s)=\prod_{P}\left(1-(1+|P|)^{-1}|P|^{-s}\right)$ and $\zeta_{\mathbb{A}}\left(\frac{1}{2}\right)$ is defined in Section $2$.

\vspace{2mm}
\noindent
Andrade \cite{AND} established the following lower bound:
\begin{theorem}[Andrade]\label{Th3}
	For every even natural number $k$, we have 
	\begin{equation*}
	\frac{1}{|\mathcal{H}_{n}|}\sum_{D\in \mathcal{H}_{n}}L(\tfrac{1}{2},\chi_{D})^{k} \gg_{k} \; n^{\tfrac{k(k+1)}{2}}.
	\end{equation*}
\end{theorem}
%The following corollary gives the lower bound of moments of $L$-functions $L(1/2, \chi_D)$ (of the conjectured order of magnitude), where $D\in \mathcal{H}_{2g+1}$.
%\begin{corollary}
%For every even natural number $k$, we have 
%	\begin{equation*}
%	\frac{1}{|\mathcal{H}_{2g+1}|}\sum_{D\in \mathcal{H}_{2g+1}}L(\tfrac{1}{2},\chi_{D})^{k} \gg_{k} \; g^{\tfrac{k(k+1)}{2}}.
%	\end{equation*}
%\end{corollary} 

%The following corollary gives the lower bound of moments of $L$-functions $L(1/2, \chi_D)$, where $D\in \mathcal{H}_{2g+2}$.
%\begin{corollary}
%For every even natural number $k$, we have 
%	\begin{equation*}
%	\frac{1}{|\mathcal{H}_{2g+2}|}\sum_{D\in \mathcal{H}_{2g+2}}L(\tfrac{1}{2},\chi_{D})^{k} \gg_{k} \; g^{\tfrac{k(k+1)}{2}}.
%	\end{equation*}
%\end{corollary} 
%\vspace{2mm}
\noindent
On the other hand,  A. Florea [\cite{FL3}, Theorem $2.7$]  found the following upper bound for a single shifted $L$-function associated with hyperelliptic curves:
\begin{theorem}[Florea]\label{florea}
	 Let $v=e^{i \theta}$, with $\theta \in [0, \pi)$. Then for every positive $k$ and any $\epsilon >0$,
	\begin{align*}
	\sum_{D\in \mathcal{H}_{2g+1}} \big|\mathcal{L}\Big(\frac{v}{\sqrt{q}},\chi_{D}\Big)\big|^{k}\ll_{k,\ve} q^{2g+1} \; g^{\ve}\, \exp\left(k\, \mathcal{M}(v,g)\,+\,\frac{k^{2}}{2}\mathcal{V}(v,g)\right), 
	\end{align*}
	
	where $\mathcal{M}(v,g)=\frac{1}{2}\log\left(\min\{g,\frac{1}{2\theta}\} \right)\quad $ and $\quad \mathcal{V}(v,g)=\mathcal{M}(v,g) + \frac{1}{2}\log g$.
\end{theorem}
%\begin{proof}
%See [\cite{FL3}, Theorem $2.7$].
%\end{proof}
%\vspace{2mm}
%\noindent
\subsubsection{Shifted moments and main results}
 In this article, for fixed $m$-tuple $\bm{k}^{(m)}=(k_{1},\ldots,k_{m})\in \mathbb{N}^{m}$, we shall investigate the following mean values of the product of $m$-shifted  quadratic Dirichlet $L$-functions: 
\begin{align}\label{Shifted product of L-functions}
\mathcal{S}_n(\bm{v}^{(m)},\bm{k}^{(m)}):= \sum_{D\in \mathcal{H}_{n}} \mathcal{L}\Big(\frac{v_{1}}{q^{\frac{1}{2}+\alpha_{1}}},\chi_{D}\Big)^{2k_1} \ldots \mathcal{L}\Big(\frac{v_{m}}{q^{\frac{1}{2}+\alpha_{m}}},\chi_{D}\Big)^{2k_m},
\end{align}
where $\bm{v}^{(m)}=(v_{1},\ldots, v_{m})\in \mathbb{C}^{m}$ with $v_j=e^{i\theta_j}, \, \theta_j\in (0, \pi]$ and $\alpha_j\in [0, \frac{1}{2})$ for $j=1, \ldots, m$. 
Also, $\theta_j=\theta_j(g)$ is a real valued function of $g$ such that $\displaystyle\lim_{g\to \infty}g |\theta_j|$ and for $i\neq j$, $\displaystyle\lim_{g \to \infty}g |\theta_i-\theta_j|$ exists or equals $\infty$. Note that one can obtain the moments of $\mathcal{L}\Big(\frac{v}{\sqrt{q}}, \chi _D\Big)$ by allowing the shifts $\alpha_j$ to tend to $0$.

%\todo[inline]{It seems we define equation $(1.3)$ in a very casual way.}
 %the moments  for the symplectic family of $L$-function over hyperelliptic ensemble.
% \vspace{0mm}
 \noindent Throughout the article, we follow that $n$ and $g$ are connected via \eqref{genus}.
Before stating our main results, let us define
\begin{align}
\label{mu}
&\mu(\bm{v}^{(m)},g)=\displaystyle\sum_{j=1}^{m}k_j \log\left(\min\Big\{\frac{1}{2|\theta_{j}|}, g \Big\} \right), \\
\label{sigma} &\sigma(\bm{v}^{(m)},g)=2\left(\sum_{j=1}^{m}k_j^2\right)\log{g}\, + \, 2\sum_{j=1}^{m}k_j^2 \log\left(\min\Big\{\frac{1}{2|\theta_{j}|}, g \Big\} \right) \\
&+\,4\displaystyle\sum_{i<j}k_i k_j \left( \log\left(\min\Big\{\tfrac{1}{ | \theta_{i}-\theta_{j}  |},g \Big\} \right)+ \log\left(\min\Big\{\tfrac{1}{ | \theta_{i}+\theta_{j}  |},g \Big\} \right)\right) \nonumber.
\end{align}
For $m=2$, we set
\begin{align}\label{notation of W}
W=\{j\in \{1, 2\}: \lim_{g\to \infty}g |\theta_j|<\infty \} \, \text{ and }\, W^{c}=\{j\in \{1, 2\}: \lim_{g\to \infty}g |\theta_j|=\infty\}. 
\end{align}
 We also define a constant depending on $W$ and  $W^c$ as 
\begin{align}\label{contant of main theorem}	
c_{\bm{v}^{(2)}}=\max \left\{\underset{g\to \infty}{\lim}\;g|\theta_{1}|,\, \underset{g\to \infty}{\lim}\;g|\theta_{2}|,\,\underset{g\to \infty}{\lim}g|\theta_{1}-\theta_{2}|, \, \underset{g\to \infty}{\lim}g|\theta_{1}+ \theta_{2}|\right\},
\end{align}
where maximum is taken only over the finite entries of the set.
\begin{remark}
If $|W|=2$ and $|W^c|=0$ then 
\[
c_{\bm{v}^{(2)}}=\max \left\{\underset{g\to \infty}{\lim}\;g|\theta_{1}|,\, \underset{g\to \infty}{\lim}\;g|\theta_{2}|,\,\underset{g\to \infty}{\lim}g|\theta_{1}-\theta_{2}|, \, \underset{g\to \infty}{\lim}g|\theta_{1}+ \theta_{2}|\right\}.
\]
If $W=\{1\}$ and $W^c=\{2\}$ then $c_{\bm{v}^{(2)}}=\lim_{g\to \infty}g|\theta_1|$.  If $W=\emptyset$ and $\displaystyle\lim_{g\to \infty}g|\theta_1-\theta_2|<\infty$, then $c_{\bm{v}^{(2)}}=\displaystyle\lim_{g\to \infty}g|\theta_1-\theta_2|$. When none of the limit is finite then $c_{\bm{v}^{(2)}}$ is an absolute constant.
\end{remark}
\noindent
We obtain a lower bound (of the conjectured order of magnitude\footnote{The conjectural order of magnitude of these $L$-functions in the  hyperelliptic ensemble can be compared with the autocorrelation of the random matrix polynomials (for example, see [\cite{CFKRS2}, Eqs. $(3.6)$ and $(4.19)$]). }) for $\mathcal{S}_n(\bm{v}^{(2)}, \bm{k}^{(2)})$ in the large degree limit i.e. when $n$ is sufficiently large and $q$ is fixed.

\begin{theorem}\label{Th1}
	Let $\bm{k}^{(2)}$ and $\bm{v}^{(2)}$ be as earlier. Assume that $\alpha_j=O\left(\frac{1}{g}\right)$ for $j=1, 2$. Then for $n$ large,  
	\begin{equation*}
	\frac{1}{|\mathcal{H}_{n}|}\sum_{D\in \mathcal{H}_n}\bigg|\mathcal{L}\Big(\frac{v_{1}}{q^{\frac{1}{2}+\alpha_{1}}},\chi_{D}\Big)\bigg|^{2k_1} \bigg|\mathcal{L}\Big(\frac{v_{2}}{q^{\frac{1}{2}+\alpha_{2}}},\chi_{D}\Big)\bigg|^{2k_2}  \gg_{\bm{k}^{(2)}, c_{\bm{v}^{(2)}}} \exp\left(\mu\left(\bm{v}^{(2)},g\right)+\frac{1}{2}\sigma\left(\bm{v}^{(2)},g\right)\right),
	\end{equation*}
	where $\mu\left(\bm{v}^{(2)},g\right)$, $\sigma\left(\bm{v}^{(2)},g\right)$ and $c_{\bm{v}^{(2)}}$ are defined by \eqref{mu}, \eqref{sigma} and \eqref{contant of main theorem} respectively.
\end{theorem}

%\begin{remark}
%\todo[inline]{Give an connection between degree and the genus of the hyperelliptic curve so that one can understand between $n$ and $\delta$.}
%\end{remark}
%The following corollary gives the lower bound of moments of $L$-functions $L(1/2, \chi_D)$ (of the conjectured order of magnitude), where $D\in \mathcal{H}_{2g+1}$.
%\begin{corollary}
%For every even natural number $k$, we have 
%	\begin{equation*}
%	\frac{1}{|\mathcal{H}_{2g+1}|}\sum_{D\in \mathcal{H}_{2g+1}}L(\tfrac{1}{2},\chi_{D})^{k} \gg_{k} \; g^{\tfrac{k(k+1)}{2}}.
%	\end{equation*}
%\end{corollary} 
In this article, we will provide the complete proof of the Theorem \ref{Th1} and from observations in the footnotes [\ref{Footnote $2$}, \ref{Footnote $5$}, \ref{Footnote $8$}], one can easily extend Theorem \ref{Th1} to the following form. 
\begin{theorem}\label{Th1.2}
	Let $\bm{k}^{(m)}$ and $\bm{v}^{(m)}$ be as earlier. Assume that $\alpha_j=O\left(\frac{1}{g}\right)$ for all $j$. Then for $n$ large,
	\begin{equation*}
	\frac{1}{|\mathcal{H}_{n}|} \sum_{D\in \mathcal{H}_{n}} \bigg|\mathcal{L}\Big(\frac{v_{1}}{q^{\frac{1}{2}+\alpha_{1}}},\chi_{D}\Big)\bigg|^{2k_1} \ldots \bigg|\mathcal{L}\Big(\frac{v_{m}}{q^{\frac{1}{2}+\alpha_{m}}},\chi_{D}\Big) \bigg|^{2k_m}\gg_{\bm{k}^{(m)}, \bm{v}^{(m)}} \exp\left(\mu\left(\bm{v}^{(m)},g\right)+\frac{1}{2}\sigma\left(\bm{v}^{(m)},g\right)\right),
	\end{equation*} 
		where $\mu\left(\bm{v}^{(m)},g\right)$ and $\sigma\left(\bm{v}^{(m)},g\right)$ are defined by \eqref{mu} and \eqref{sigma} respectively.
\end{theorem}

\noindent
We also establish an upper bound of nearly the conjectured order of magnitude for the sum $\mathcal{S}_n(\bm{v}^{(m)}, \bm{k}^{(m)})$.

\begin{theorem}\label{Th2} With the assumption as in the Theorem \ref{Th1.2}, for any $\epsilon>0$,
	\begin{equation*}
	\frac{1}{|\mathcal{H}_{n}|}\sum_{D\in \mathcal{H}_{n}} \bigg|\mathcal{L}\Big(\frac{v_{1}}{q^{\frac{1}{2}+\alpha_{1}}},\chi_{D}\Big)\bigg|^{2k_1} \ldots \bigg|\mathcal{L}\Big(\frac{v_{m}}{q^{\frac{1}{2}+\alpha_{m}}},\chi_{D}\Big) \bigg|^{2k_m}\ll_{\bm{k}^{(m)}, \varepsilon} n^{\ve} \exp\left(\mu\left(\bm{v}^{(m)},g\right)+\frac{1}{2}\sigma\left(\bm{v}^{(m)},g\right)\right),
	\end{equation*} 
	where $\mu\left(\bm{v}^{(m)},g\right)$ and $\sigma\left(\bm{v}^{(m)},g\right)$ are defined by \eqref{mu} and \eqref{sigma} respectively.
	%where
	%%&\mu(\textbf{u}^{(m)},g)=\displaystyle\sum_{j=1}^{m}k_j \log5\left(\min\Big\{\frac{1}{2|\theta_{j}|}, \delta \Big\} \right) \, \, \text{ and } \\
%	&\sigma(\textbf{u}^{(m)},\delta)=2\left(\sum_{j=1}^{m}k_j^2\right)%\log{\delta}\, + \, 2\sum_{j=1}^{m}k_j^2 \log\left(\min\Big\{\frac{1}{2|%\theta_{j}|}, \delta \Big\} \right) \\
%&+\,4\displaystyle\sum_{i<j}k_i k_j \left( \log\left(\min\Big\{\tfrac{1}{ | \theta_{i}-\theta_{j}  |},\delta \Big\} \right)+ \log\left(\min\Big\{\tfrac{1}{ | \theta_{i}+\theta_{j}  |},\delta \Big\} \right)\right).
	%\end{align*}
		
\end{theorem}

%\begin{theorem}\label{Th3}
	%Let $\vec{u}=(u_{1},\ldots,u_{m})$ with $ u_{j}=e^{i\,\theta_{j}}$ and $\theta_{j}\in [0,\pi)$, for $j=1,\ldots, m$. Assume that $\alpha_{j}\leq \frac{1}{4\delta}$, where $\delta$ is define by \eqref{delta}. For every even number $m$ and $\epsilon>0$, we have
	%\begin{equation*}
%%\end{equation*} 
	%where
	%\begin{align*}
	%\mu_{0}(\textbf{u},\delta)=\displaystyle\sum_{\substack{i=1\\ i \text{ odd}}}^{m-1}\log\left(\min\Big\{\tfrac{1}{|\theta_{i}-\theta_{i+1}|},\delta \Big\} \right) \, \, \text{ and } \, \, \sigma_{0}(\textbf{u},\delta)=m\log{\delta}\, +\,2\displaystyle\sum_{i<j}\log\left(\min\Big\{\tfrac{1}{ | \theta_{i}-\theta_{j}  |},\delta \Big\} \right).
	%\end{align*}
%\end{theorem}	

% $\textbf{odd, even er case ta ki bolochile moneni ???  tik kore diyo.}$
 
% \begin{corollary}\label{Co2}
	%\begin{align*}
	%\frac{1}{|\mathcal{H}_{2g+1}|}\sum_{D\in \mathcal{H}_{2g+1}}|\mathcal{L}(q^{-1/2},\chi_{D})|^{k}\ll_{\ve, k} g^{\tfrac{k(k+1)}{2}+\ve}.
	%\end{align*}
%\end{corollary}
\noindent
Theorem \ref{Th2} is also true for any fixed sequence $\bf{k}^{(m)}$ of positive real numbers.
 Let us define $\mathcal{L}^{(l)}(u,\chi_{D})$ as the $l$-th derivative of $\mathcal{L}(u,\chi_{D})$.  As an important  consequence of Theorem \ref{Th2}, we have the following upper bound.
 \begin{theorem}\label{th5}
 	Let $l\in \mathbb{N}$ and $\varepsilon >0$. For $n$ large, we have
 	\begin{equation*}
 	\sum_{D\in \mathcal{H}_{n}}\big|\mathcal{L}^{(l)}\Big(q^{-1/2},\chi_{D}\Big) \big|^{k}\ll_{k, l, \ve} |\mathcal{H}_{n}| g^{\frac{1}{2}k(k+1)+lk+\varepsilon}.
 	\end{equation*} 	
 
 \end{theorem}

\subsubsection{Applications}
From the Theorem \ref{Th2} and in light of \eqref{delta} if we specialize $n$
as $n = 2g + 1 $ and $\alpha_j$'s are zero then we recover Theorem \ref{florea}. 

\noindent
The case of the mean value for
$\mathcal{L}(q^{-1/2}, \chi_D)$ taken over $\mathcal{H}_{2g+2}$ was investigated by Jung \cite{JUNG}. Taking $n=2g+2$, we have the following corollary which generalizes the Theorem \ref{florea} :
\begin{corollary}
Let $\varepsilon >0$. For $n$ large, we have
	\begin{align*}
	\frac{1}{|\mathcal{H}_{2g+2}|}\sum_{D\in \mathcal{H}_{2g+2}}|\mathcal{L}(q^{-1/2},\chi_{D})|^{k}\ll_{k, \ve} g^{\frac{1}{2}k(k+1)+\ve}.
	\end{align*}
\end{corollary}

Similarly, Theorem \ref{th5} provide upper bound for $k$-th moment of $\mathcal{L}^{(m)}(q^{-1/2},\chi_{D})$ with $D\in  \mathcal{H}_{2g+1}$ and $D\in  \mathcal{H}_{2g+2}$.
More precisely,
\begin{corollary}
	Let $l\in \mathbb{N}$ and $\varepsilon >0$. For $n$ large, we have
 	\begin{equation*}
 	\sum_{D\in \mathcal{H}_{2g+1}}\big|\mathcal{L}^{(l)}\Big(q^{-1/2},\chi_{D}\Big) \big|^{k}\ll_{k, l, \varepsilon} q^{2g+1} g^{\frac{1}{2}k(k+1)+lk+\varepsilon},
 	\end{equation*} 	
 	\begin{equation*}
 	\sum_{D\in \mathcal{H}_{2g+2}}\big|\mathcal{L}^{(l)}\Big(q^{-1/2},\chi_{D}\Big) \big|^{k}\ll_{k, l, \varepsilon} q^{2g+2} g^{\frac{1}{2}k(k+1)+lk+\varepsilon}.
 	\end{equation*} 
\end{corollary}
Theorem \ref{Th1.2} and \ref{Th2}  can be related to the main results of \cite{CHA} and \cite{MU}.  In fact we can say that $\mathcal{L}\Big(\frac{v_{1}}{q^{\frac{1}{2}+\alpha_{1}}},\chi_{D}\Big)$ and $\mathcal{L}\Big(\frac{v_{2}}{q^{\frac{1}{2}+\alpha_{2}}},\chi_{D}\Big)$ are essentially correlated when $|\theta_j|\asymp \frac{1}{g}$ for $j=1,2$ and independent when one of  $\theta_j$'s  is much larger than $\frac{1}{g}$. More precisely, we have the following corollaries.

%\vspace{2mm}
%% As a direct consequence of the Theorem \ref{florea} and Theorem \ref{Th2} we have the following upper for the moments of quadratic Dirichlet $\mathcal{L}$-function over function field in hyperelliptic ensemble.
 
 \begin{corollary}\label{Cor1} Let $W$ and $W^c$ be defined by \eqref{notation of W}. For every $\varepsilon>0$ and $n$ large, 
 \begin{align*}
 	&\sum_{D\in \mathcal{H}_{n}} \bigg|\mathcal{L}\Big(\frac{v_{1}}{q^{\frac{1}{2}+\alpha_{1}}},\chi_{D}\Big) \mathcal{L}\Big(\frac{v_{2}}{q^{\frac{1}{2}+\alpha_{2}}},\chi_{D}\Big) \bigg|^{2k}\\
 &	\ll_{k, \varepsilon} \left\{
 	\begin{array}
 	[c]{ll}
 	|\mathcal{H}_{n}| \; \, {g}^{2k(4k+1)+\ve} & \text{if \;  } |W|=2,
 	\vspace{1mm}
 	\\
 	\frac{|\mathcal{H}_{n}|\,\, {g}^{3k^2+k+\ve}}{|\theta_2|^{k^2+k}|\theta_1- \theta_2|^{2k^2}|\theta_1 + \theta_2|^{2k^2}} & \text{if \; } W=\{1\}, W^c=\{2\},
 	\vspace{2mm}
 	\\
 	  \frac{|\mathcal{H}_{n}|\,\, {g}^{2k^2+\ve}}{|\theta_1\theta_2|^{k^2+k}|\theta_1- \theta_2|^{2k^2}|\theta_1 + \theta_2|^{2k^2}}     & \text{if \,} |W^c|=2,\, \displaystyle\lim_{g\to \infty}g|\theta_1-\theta_2|=\infty, \vspace{2mm}
 	  \\
 	  \frac{|\mathcal{H}_{n}|\,\, {g}^{4k^2+\ve}}{|\theta_1\theta_2|^{k^2+k}|\theta_1 + \theta_2|^{2k^2}} & \text{if \,}  |W^c|=2,\, \displaystyle\lim_{g\to \infty}g|\theta_1-\theta_2|<\infty,
 	\end{array}
 	\right.
 	\end{align*}
 	%where $u_j$'s are as in Theorem \ref{Th2}.
 \end{corollary}
  \begin{corollary}\label{Cor2} 
  	With the assumption as in Corollary \ref{Cor1}, for $n$ large, we have
 \begin{align*}
 	&\sum_{D\in \mathcal{H}_{n}} \bigg|\mathcal{L}\Big(\frac{v_{1}}{q^{\frac{1}{2}+\alpha_{1}}},\chi_{D}\Big) \mathcal{L}\Big(\frac{v_{2}}{q^{\frac{1}{2}+\alpha_{2}}},\chi_{D}\Big) \bigg|^{2k}\\ 
 &\gg_{k, \bm{u}^{(2)}} \left\{
 	\begin{array}
 	[c]{ll}
 	|\mathcal{H}_{n}| \; \, {g}^{2k(4k+1)+\ve} & \text{if \;  } |W|=2,
 	\vspace{1mm}
 	\\
 	\frac{|\mathcal{H}_{n}|\,\, {g}^{3k^2+k+\ve}}{|\theta_2|^{k^2+k}|\theta_1- \theta_2|^{2k^2}|\theta_1 + \theta_2|^{2k^2}} & \text{if \; } W=\{1\}, W^c=\{2\}, ,
 	\vspace{2mm}
 	\\
 	  \frac{|\mathcal{H}_{n}|\,\, {g}^{2k^2+\ve}}{|\theta_1\theta_2|^{k^2+k}|\theta_1- \theta_2|^{2k^2}|\theta_1 + \theta_2|^{2k^2}}     & \text{if \,} |W^c|=2,\, \displaystyle\lim_{g\to \infty}g|\theta_1-\theta_2|=\infty, \vspace{2mm}
 	  \\
 	  \frac{|\mathcal{H}_{n}|\,\, {g}^{4k^2+\ve}}{|\theta_1\theta_2|^{k^2+k}|\theta_1 + \theta_2|^{2k^2}} & \text{if \,}  |W^c|=2,\, \displaystyle\lim_{g\to \infty}g|\theta_1-\theta_2|<\infty.
 	\end{array}
 	\right.
 	\end{align*}
 	%where $u_j$'s and $g$ are as in Theorem \ref{Th1.2}.
 \end{corollary}
 \begin{remark}
  Corollary \ref{Cor1} and \ref{Cor2} gives the lower and upper bound over all monic square-free polynomials of both even and odd degree near the critical line.
 \end{remark}
 \section{Background for $L$-functions over function fields}\label{prilimimaries}
We begin this section with some preliminaries of $L$-functions over function fields. We will use \cite{ROS} as a general reference.
\subsection{Basic facts on $\mathbb{F}_q[t]$} We start by fixing a finite field $\mathbb{F}_{q}$ of odd cardinality $q=p^r$, $r\geq 1$ with a prime $p$. We denote by $\mathbb{A}=\mathbb{F}_{q}[t]$ the polynomial ring over $\mathbb{F}_{q}$.
 For a polynomial $f$ in $\mathbb{F}_{q}[t]$, it's degree will be denoted by either $\deg(f)$ or $d(f)$. 
 
 \vspace{1mm}
\noindent 
 The set of all monic polynomials and monic irreducible polynomials of degree $n$ are denoted by $\mathcal{M}_{n,q}$ (or simply $\mathcal{M}_{n}$ as we fix $q$) and $\mathcal{P}_{n, q}$ (or simply $\mathcal{P}_{n}$) respectively. Let $\mathcal{M}=\cup_{n\geq 1} \mathcal{M}_{n}$ and $\mathcal{P}=\cup_{n\geq 1} \mathcal{P}_{n}$. we also denote the set of all monic polynomials and monic irreducible polynomials of degree less or equal to $n$ by $\mathcal{M}_{\leq n,q}$ (or simply $\mathcal{M}_{\leq n}$) and $\mathcal{P}_{\leq n,q}$ (or simply $\mathcal{P}_{\leq n}$) respectively. 
Let $\mathcal{H}_{n}$ denotes the set of monic square-free polynomials of degree $n$. Observe that for $n\geq 1$, $|\mathcal{M}_{n}|=q^{n}$ and 
$$|\mathcal{H}_{n}|=\left\{\begin{array}
[c]{ll}q,&\; \text{if}\; n=1,\\
q^{n-1}(q-1),&\; \text{if}\; n\ge 2.
\end{array}
\right.$$ If $f$ is is a non-zero polynomial $\mathbb{F}_{q}[t]$, we define the norm of $f$ to be $|f|=q^{d(f)}$. If $f=0$, we set $|f|=0$.
The prime polynomial theorem (see \cite{ROS}, Theorem $2.2$) states that 
\begin{align}\label{prime poly th}
|\mathcal{P}_{n,q}|=\frac{q^{n}}{n} \,\,+\,\,O\Big(\frac{q^{\frac{n}{2}}}{n}\Big).
\end{align}
\noindent
 The zeta function of $\mathbb{A}$, denoted by $\zeta_{\mathbb{A}}(s)$ and is defined by $$\zeta_{\mathbb{A}}(s):=\displaystyle\sum_{f\in \mathcal{M}}\frac{1}{|f|^{s}}=\displaystyle\prod_{P\in \mathcal{P}} \left( 1- |P|^{-s} \right)^{-1},\quad\quad  \Re(s)>1.$$ One can easily prove that $\zeta_{\mathbb{A}}(s)=\frac{1}{1-q^{1-s}}$, and this provides an analytic continuation of zeta function to the complex plane with a simple pole at $s=1$. Using the change of variable $u=q^{-s}$, 
 \begin{align*}
 \mathcal{Z}(u)=\sum_{f\in \mathcal{M}}u^{d(f)}=\frac{1}{1-qu}, \quad \text{ if } |u|<\frac{1}{q}.
 \end{align*} 
 \noindent
\subsection{Quadratic Dirichlet character and properties of their $L$-functions} For a monic irreducible polynomial $P$, the quadratic residue symbol $\left(\frac{f}{P} \right)$ is defined by 
 \begin{align*}
  \left( \frac{f}{P}\right)= 
 \left\{
 \begin{array}
 [c]{ll}
  1, & \text{if\, $f$ is a square\;} (\text{mod}\,\,  P),\,\, P\nmid f \\
 -1, & \text{if\, $f$ is not a square\;} (\text{mod}\,\, P),\,\, P\nmid f\\
  0, & \text{if \; } P\mid f.
  \end{array}
 \right.
 \end{align*}
 For monic square-free polynomial $D\in \mathbb{F}_{q}[t]$, the symbol $\left( \frac{D}{.}\right)$ is defined by extending the above residue symbol multiplicatively. We denote the quadratic Dirichlet character $\chi_{D}$ by $$\chi_{D}(f)= \left(\frac{D}{f} \right).$$
 \noindent
 The $L$-function associated to the quadratic Dirichlet character $\chi_{D}$ is defined by
 \begin{align*}
 L(s,\chi_{D})=\sum_{f\in \mathcal{M}} \frac{\chi_{D}(f)}{|f|^{s}}=\prod_{P\in \mathcal{P}}\left(1-\chi_{D}(P)\,|P|^{-s} \right)^{-1},\;\; \Re(s)>1.
 \end{align*} 
Using the change of variable $u=q^{-s}$, we have  
 	\begin{align*}
 	\mathcal{L}(u,\chi_{D})=\sum_{f\in \mathcal{M} } \chi_{D}(f)\, u^{d(f)}=\prod_{P\in\mathcal{P}}\left(1-\chi_{D}(P)\,u^{d(P)} \right)^{-1},\;\;\, |u|<\frac{1}{q} .
 	\end{align*}
 	\noindent
 By [\cite{ROS}, Proposition $4.3$], we see that if $n\geq d(D)$ then $$\sum_{f\in \mathcal{M}_{n}}\chi_{D}(f)=0 .$$ It implies that $\mathcal{L}(u,\chi_{D})$ is a polynomial of degree at most $d(D)-1$. From \cite{Rud}, $\mathcal{L}(u, \chi_D)$ has a trivial zero at $u=1$ if and only if $d(D)$ is even. This allows us to define the completed $L$-function as
\[
L(s,\chi_{D})=\mathcal{L}(u, \chi_D)=(1-u)^{\lambda}\mathcal{L}^{*}(u, \chi_D)=(1-q^{-s})^{\lambda}{L}^{*}(s, \chi_D),
\]
where 
\begin{align}\label{definition of lambda}
  \lambda= 
 \left\{
 \begin{array}
 [c]{ll}
  1, & \;\text{if\, $d(D)$ even}, \\
 0, &\; \text{if\, $d(D)$ odd},
  \end{array}
 \right.
\end{align}
and $\mathcal{L}^*(u, \chi_D)$ is a polynomial of degree 
\begin{align}\label{delta}
2g=d(D)-1-\lambda
\end{align}
satisfying the functional equation 
\[
\mathcal{L}^*(u, \chi_D)=(qu^2)^{g}\mathcal{L}^*\left(\frac{1}{qu}, \chi_D\right)
\]
Because $\mathcal{L}$ and $\mathcal{L}^*$ are polynomial in $u$, it is convenient to define 
\[
L^*(s, \chi_D)=\mathcal{L}^*(u, \chi_D)
\]
so that the above functional equation can be rewritten as
\[
L^*(s, \chi_D)=q^{(1-2s)g}L^*(1-s, \chi_D).
\]
\noindent
The Riemann hypothesis for curve over finite fields, established by Weil \cite{WEIL}, asserts that all the non-trivial zero of $\mathcal{L}^{*}(u, \chi_{D})$ are lie on the circle $|u|=q^{-1/2}$, i.e,
 \[ 
 \mathcal{L}^{*}(u, \chi_{D})=\prod_{j=1}^{2g}\left(1-u \, \nu_{j}\, \right)\;\; \text{with}\;\, |\nu_{j}|=\sqrt{q}\;\, \text{for all}\,\, j.  
 \]

\noindent
One can define the completed $L$-function in the following way. Set 
\begin{align}\label{Gamma factor of the completed l functions}
X_D (s)=|D|^{\frac{1}{2}-s}X(s),
\end{align}
 where
\begin{align*}
X(s)=\left\{
 \begin{array}
 [c]{ll}
  q^{s-\frac{1}{2}}, &\, \text{if\, $d(D)$ odd} \\
 \frac{1-q^{-s}}{1-q^{-(1-s)}}q^{-1+2s}, &\, \text{if\, $d(D)$ even}.
  \end{array}
 \right.
\end{align*}
Let us consider
\begin{align}\label{Lambda}
\Lambda(s, \chi_D)=L(s, \chi_D)X_D(s)^{-\frac{1}{2}}.
\end{align}
\noindent
Then $\Lambda(s, \chi_D)$ satisfies the symmetric functional equation 
\begin{align}\label{func.equ}
\Lambda(s, \chi_D)=\Lambda(1-s, \chi_D).
\end{align}
\noindent
\subsection{Spectral Interpretation} 
Let $C$ be a non-singular projective curve over $\mathbb{F}_q$ of genus $g$. For each extension field of
degree $k$ of $\mathbb{F}_q$, denote by $N_k(C)$ the number of points of $C$ in $\mathbb{F}_{q^k}$. Then, the zeta
function associated to $C$ defined as
\[
Z_C(u)=\exp\left(\sum_{k=1}^\infty N_k(C)\frac{u^k}{k}\right), \quad |u|<\frac{1}{q},
\]
is known to be a rational function of $u$ of the form
\[
Z_C(u)=\frac{P_C(u)}{(1-u)(1-qu)}.
\]
Additionally, we know that $P_C(u)$ is a polynomial of degree $2g$ with integer coefficients, satisfying
a functional equation
\[
P_C(u)=(qu^2)^g
P_C\left(\frac{1}{qu}\right).
\]

\noindent
The Riemann Hypothesis, proved by Weil \cite{WEIL}, says that the zeros of $P_C(u)$ all lie on the circle $|u| =\frac{1}{\sqrt{q}}$. Thus one may give a spectral interpretation of
$P_C(u)$ as the characteristic polynomial of a $2g \times 2g$ unitary matrix $\Theta_C$:
\[
P_C(u)=\det \left(I- u\sqrt{q}\Theta_C\right).
\]
Thus the eigenvalues $e^{i\theta_j}$ of $\Theta_C$ correspond to the zeros, $q^{-1/2}e^{-i\theta_j}$, of $Z_C(u)$.
The matrix $\Theta_C$ is called the unitarized  Frobenius class of $C$.

\vspace{1mm}
\noindent
To put this in the context of our case, note that, for a family of hyperelliptic curves $C_D:\, y^2=D(t)$ of genus $g$, the numerator of the zeta function $Z_C(u)$  associated to $C_D$ is coincide with the $L$-function $\mathcal{L}^*(u, \chi_D)$, i.e., $P_C(u)=\mathcal{L}^*(u, \chi_D)$.

\section{Preliminary Lemmas}  
We start with an analog of approximate functional equation for $L(s,\chi_{D})$. 
Recall that $2g=n-1-\lambda$ where $\lambda$ is defined as in \eqref{definition of lambda}.
\begin{lemma}[Approximate functional equation]\label{Le1} 
Let $\chi_D$ be a quadratic Dirichlet character, where $D\in \mathcal{H}_n$. Then for $1/2\leq s< 1$, 
	\begin{align*}
	L(s,\chi_{D})&=\sum_{f\in \mathcal{M}_{\leq g}}\frac{\chi_{D}(f)}{|f|^{s}}\; + X_D(s)\, \sum_{f\in \mathcal{M}_{\leq g-1}}\frac{\chi_{D}(f)}{|f|^{1-s}}\\
	&-\lambda q^{-s(g+1)}\sum_{f\in \mathcal{M}_{\leq g}}\chi_D(f)- \lambda \, X_D(s)\, q^{-(1-s)g} \sum_{f\in \mathcal{M}_{\leq g-1}}\chi_D(f),
	\end{align*} 	
	where $X_D(s)$ is defined by \eqref{Gamma factor of the completed l functions} respectively.
\end{lemma}

\begin{proof} 
The case $s=\frac{1}{2}$ is proved in \cite{AK1} for $D\in \mathcal{H}_{2g+1}$ and \cite{JUNG} for $D\in \mathcal{H}_{2g+2}$. Their methods can be easily generalized for any $s\in (1/2, 1)$.
\end{proof}
\noindent
The following lemma gives an asymptotic formula for a square polynomial in hyperelliptic ensemble.

\begin{lemma}\label{Le2}
For $f\in \mathcal{M}$, we have 
\begin{equation*}
\frac{1}{|\mathcal{H}_{n}|}\sum_{D\in \mathcal{H}_{n}}\chi_{D}(f^{2})\; =\prod_{\substack{P\in\mathcal{P}\\P\mid f}}\left(1+\frac{1}{|P|} \right)^{-1} \; +\; O(|\mathcal{H}_n|^{-1}). 
\end{equation*}
\end{lemma}

\begin{proof}
See [\cite{BF1}, Lemma $3.7$] for $n=2g+1$.  To get the result for $n=2g+2$, it is a small adaptation of their proof.
\end{proof}	
\noindent
The following lemma is an analog of Polya-Vinogradov inequality over function fields.
\begin{lemma}[Polya-Vinogradov inequality]\label{Le3}
	For $l\in \mathcal{M}$ not a perfect square, let $l=l_{1}{l_{2}}^{2}$ with $l_{1}$ square-free. Then for any $\epsilon>0$, 
\begin{equation*}
\bigg|\sum_{D\in \mathcal{H}_{n}}\chi_{D}(l)\bigg|\;\ll_{\ve}\; \sqrt{|\mathcal{H}_n|}|l_{1}|^{\epsilon}. 
\end{equation*} 
\end{lemma}

\begin{proof}
One can easily generalize the above inequality which was proved in [\cite{BF2}, Lemma $3.5$] for $n=2g+1$. Here we give a different proof in the above form for completeness.\\
First assume that $l_1=P_1P_2\ldots P_k$, where $P_j$'s are distinct prime polynomials, and $\deg (l_1)\leq n$. Similar to the proof of Lemma $3.5$ in \cite{BF1}, which in particular case $k=2$, one can show:
\[
\bigg|\sum_{D\in \mathcal{H}_n}\chi_D(l)\bigg|=\bigg|\sum_{D\in \mathcal{H}_n}\chi_D(l_1)\bigg|\leq \frac{qg^{k-1}\left(d(P_1)+\ldots+d(P_k)\right)}{d(P_1)\ldots d(P_k)}|l_1|^{\frac{1}{2}}\ll_{\epsilon} \sqrt{|\mathcal{H}_n|}|l_1|^{\epsilon}.
\]
Finally let $\deg(l_1)>n$. We combine Lemma $3.1$ of \cite{BF2} and Lemma $3.5$ of \cite{BF1}  to obtain
$$
\bigg|\sum_{D\in \mathcal{H}_n}\chi_D(l_1)\bigg|\ll_{\epsilon} \sqrt{|\mathcal{H}_n|}|l_1|^{\epsilon}.
$$
\end{proof}
The following lemma gives an upper bound for the logarithm of $\mathcal{L}(u,\chi_D)$ inside the critical region. 
\begin{lemma}\label{Le8}
	Let $0\leq \alpha\leq \tfrac{1}{2}$, $v=e^{i\theta}$, $\theta\in [0,\pi)$ and $N$ be a positive integer. Then for $D\in \mathcal{H}_{n}$,
   \begin{align*}
	\log{\Big|\mathcal{L}\Big(\frac{v}{q^{\frac{1}{2}+\alpha}},\chi_{D}\Big)\Big|}\leq \tfrac{2g}{N+1}\log\left(\frac{1+q^{-\alpha(N+1)}} {1+q^{-2(N+1)}} \right) + \Re \sum_{d(f)\leq N}\frac{a_{\alpha}\left(d(f)\right) \chi_{D}(f) \,\Lambda(f) v^{ d(f)}} {|f|^{\frac{1}{2}}} +O(1),	
   \end{align*} 
  where
     \begin{align*}
  a_{\alpha}(d(f))=\frac{1}{d(f)|f|^{\alpha}} -\frac{1}{d(f)|f|^{2}} \;+\;  O\left(\frac{1}{(N+1)q^{(N+1)\alpha}}\right), \quad \text{ for } 1\leq d(f)\leq N.
   \end{align*}
%    Furthermore, 
%    \[
%    \Big|\mathcal{L}\left(q^{-\frac{1}{2}-\alpha},\chi_{D}\right)\Big|\ll \exp\left(\frac{\delta^{1-2\alpha}}{\log_{q}{\delta}} \right).
%    \]	
\end{lemma}
\begin{proof}

From the functional equation \eqref{func.equ}, we observe that 
  \[
  \Big|\Lambda(\alpha+it,\chi_{D})\Big|=\bigg|\frac{\Lambda\left(\frac{5}{2}-it,\chi_{D}\right) \Lambda(1-\alpha-it,\chi_{D})}{\Lambda\left(-\frac{3}{2}+it,\chi_{D}\right)}\bigg|=\frac{|\Lambda\left(\frac{5}{2}-it,\chi_{D}\right)| |\Lambda(1-\alpha+it,\chi_{D})|}{|\Lambda\left(-\frac{3}{2}+it,\chi_{D}\right)|}.
  \]
  Recall that
  \[
  L(\alpha+it,\chi_{D})= \left(1-q^{-\alpha-it}\right)^{\lambda}{L}^{*}(\alpha+it,\chi_{D}).
  \]
  %where $\lambda$ is defined by \eqref{definition of lambda}.
  Note that $\Big|{L}^{*}\left(\frac{5}{2}-it,\chi_{D}\right)\Big|\sim 1$. Using the expression \eqref{Lambda} for $\Lambda(s,\chi_{D})$, we get \begin{align*}
 \big|{L}(\alpha+it,\chi_{D})\big|=q^{g(5-2 \alpha)} \big|1-q^{-\alpha- it}\big|^{\lambda} \prod_{j=1}^{2g} \left(\frac{q^{2\alpha-1} +1-2q^{\alpha-\frac{1}{2}}\cos(2\pi\theta_{j}-t\log{q})}{q^{4} +1-2q^{2}\cos(2\pi\theta_{j}-t\log{q})}\right)^{\frac{1}{2}},
 \end{align*}
  Since
 \[
 q^{2\alpha-1} +1-2q^{\alpha-\frac{1}{2}}\cos(2\pi\theta_{j}-t\log{q})=(q^{\alpha-\frac{1}{2}}-1)^{2}+ 4q^{\alpha-\frac{1}{2}}{\sin}^{2}\left(\pi\theta_{j}-\frac{t\log{q}}{2}\right)
 \]
  with a similar expression holding for the denominator, it follows that
 \begin{align*}
  \log\left|L(\alpha+it,\chi_{D})\right|= g\left(\frac{5}{2}-\alpha\right)\log{q} -\frac{1}{2}\sum_{j=1}^{2g}\log\left(\frac{a^{2}+{\sin}^{2}(\pi\theta_{j}-\frac{t\log{q}}{2})}{b^{2}+{\sin}^{2}(\pi\theta_{j}-\frac{t\log{q}}{2})}\right) +O(1),
 \end{align*}
 where $$a=\frac{q^{2}-1}{2q},\; b=\frac{q^{\alpha-\frac{1}{2}}-1}{2q^{\frac{\alpha}{2}-\frac{1}{4}}}.$$
 The remaining part of the proof is the same as the proof of Lemma $8.1$ in \cite{FL3} proved by A. Florea.
\end{proof}

\begin{lemma}\label{Le9}
	Let $\theta\in (-\pi, \pi)$, then we have\,\, $\displaystyle\sum_{m=1}^{n}\frac{\cos(\theta m)}{m}\leq \log\left(\min\Big\{n,\frac{1}{|\theta|}\Big\} \right)+ O(1).$
\end{lemma}
\begin{proof}
	See [\cite{FL3}, Lemma $9.1$].
\end{proof}	
\begin{lemma}\label{Le10}
	Let $k,\, y$ be integers such that $2ky\leq n$. For any complex numbers $\{a(P)\}_{p\in \mathcal{P}}$, we have
	\begin{align*}
	\sum_{D\in \mathcal{H}_{n}}\bigg|\sum_{d(P)\leq y}\frac{a(P)\chi_{D}(P)}{|P|^{\frac{1}{2}}}\bigg|^{2k}\ll \, |\mathcal{H}_n| \,\, \frac{(2k)!}{k!\, 2^{k}} \left(\sum_{d(P)\leq y}\frac{|a(P)|^{2}}{|P|} \right)^{k}. 
	\end{align*}
\end{lemma}
\begin{proof}
This is an easy generalization of the Lemma $8.4$ of \cite{FL3} and Lemma $6.3$ of \cite{SY}.		 
\end{proof} 

During the study of our main theorems it seems interesting to estimate the following bounds for the zeta function over function fields. This is an analog of bounding the Riemann zeta function near to $1$-line.
\begin{lemma}\label{Le11}
	Let $v=e^{ i\theta}$, where $\theta\in (-\pi, \pi)$. %We define 
%	\[
%	W=\{ \theta \, : \, \lim_{g\to \infty} g\,|\theta|< \infty\} \quad \text{ and } \quad W^{c} =\{ \theta \, : \, \lim_{g\to \infty} g\,|\theta| = \infty\}.
%	\]
	%We assume that $\theta_j:=\theta_j(g)$ is a function of a integer variable $g$ with $\theta_j=o(1)$ for sufficiently large $g$. 
	Let $C$ be a circle of radius $\frac{\widetilde{r}}{g}$ centred at $\frac{1}{q}$,
	where 
	\[
	\widetilde{r}=\lim_{g\to \infty}g\, |\theta|< \infty.
	\]
For any $u$ in $C$,  we have 
	 $$\mathcal{Z}(uv)\ll g\quad \text{if}\;\,\lim_{g\to \infty}g\, |\theta|< \infty.$$
For any $u$ such that $|u-\frac{1}{q}|=O(1/g)$, we have
	   \[
	  \mathcal{Z}(uv)\ll \frac{1}{|\theta|}\quad \text{if}\;\,\lim_{g\to \infty}g\, |\theta|=\infty. 
	  \]

\end{lemma}
%\todo[inline]{We have to rewrite the statement correctly and correct the proof accordingly.}
\begin{proof}
	%We know that 
	%%%\]
	
	%Now by using $1-{u_{j}}^2= -2i\theta_{j}\left(1+\frac{2i\theta_{j}}{2!}+\dots\right)$, 
	%$|1-qu|\le \frac{q}{g}$ and $\frac{1}{|\theta_{j}|}=\b{o}(g)$,
	\noindent
	First assume that $\theta\in(-\pi,\,\pi)$ be such that $\displaystyle\lim_{g\to \infty}g\, |\theta|=\infty$ . Then using $|u-\frac{1}{q}|=O\left(1/g\right)$, we have the following estimates:
	 \[
	 \bigg|\frac{v}{(1- v)}(1-qu) \bigg|=o(1)
	 \]
	  and 
\[
|(1- v)^{-1}|\ll \frac{1}{|\theta|}.
\]
\noindent
Thus
\[
|\mathcal{Z}(uv)|=|(1-quv)^{-1}|=\bigg|(1- v)^{-1} \left( 1+ \frac{v}{(1- v)}(1-qu)\right)^{-1}\bigg|\ll \frac{1}{|\theta|}.
\]

\noindent	
Finally let $\theta$ be such that $\displaystyle\lim_{g\to \infty}g\, |\theta|<\infty$. Then $|u-\frac{1}{q}|\leq \frac{\widetilde{r}}{g}$.	
 We use the change of variable $u=q^{-s}$ to get the hypothesis of the form $|s-1|\leq \frac{\widetilde{r}}{g}$.
 %Since $|\theta_{j}|\le \frac{1}{2g}$ and $|u-\frac{1}{q}|\leq \frac{1}{g}$ with $u=q^{-s}$ , this implies $|s-1|\le \frac{1}{g}$.
Since  \[
 \mathcal{Z}(uv)=\sum_{f \in \mathcal{M}} (uv)^{\deg(f)},
 \] 
 it is enough to show that  
  \[
  \sum_{f \in \mathcal{M}} \frac{1}{|f|^{1+\,\widetilde{r}/g\, -\,i\theta/\log q}}= O(g). %\prod_{P\in \mathcal{P}}\left(1-\frac{1}{|P|^{1+\frac{\widetilde{c}}{g}\, -\,i\theta/\log q}}\right)^{-1}.
  \]
 Therefore using Lemma \ref{Le9} and the prime polynomial theorem, we obtain
 \begin{align*}
 \log\bigg|\sum_{f \in \mathcal{M}} \frac{1}{|f|^{1+\,\widetilde{r}/g\, -\,i\theta/\log q}}\bigg|=&\Re \sum_{P\in \mathcal{P}} \frac{1}{|P|^{1+\,\widetilde{r}/g\, -\,i\theta/\log q}}\, +\, O(1)= \Re{\sum_{n}\frac{1}{nq^{n(\widetilde{r}/g\, -\,i\theta/\log q)}}} + \, O(1)\\
 =& \Re{\sum_{n\le g}\frac{q^{\frac{in\theta}{\log q}}}{n}}-\Re{\sum_{n\le g}\left(\frac{1}{n}-\frac{1}{nq^{\frac{\widetilde{r}n}{g}}}\right)q^{\frac{in\theta}{\log q}}} + \Re{\sum_{n>g}\frac{q^{\frac{in\theta}{\log q}}}{nq^{\frac{\widetilde{r}n}{g}}}} +O(1)\\
 =&{\sum_{n\le g}\frac{\cos(n\theta)}{n}} +O(1) \le \; \log\left(\min\left\{g,\frac{1}{|\theta|} \right\} \right)  \le \; \log {g},
 \end{align*}
and the lemma's proof is concluded.
\end{proof}

\section{Proof of Theorem \ref{Th1}}
Throughout this section, for  the sake of simplicity, we write $\bm{v}^{(2)}$ and $\bm{k}^{(2)}$ simply as $\bm{v}$ and $\bm{k}$ respectively.
 For any $k_{1}, k_2 \in \mathbb{N}$, we write 
\begin{align}\label{product of l functions}
{\mathcal{L}\Big(\frac{v_{1}}{q^{1/2+\alpha_{1}}},\chi_{D}\Big)}^{k_{1}}{\mathcal{L}\Big(\frac{v_{2}}{q^{1/2+\alpha_{2}}},\chi_{D}\Big)}^{k_{2}}=\sum_{f\in\mathcal{M} }a_f \frac{\chi_{D}(f)}{|f|^{1/2}} ,
\end{align}
where
\begin{align}\label{definition of main coeeficient}
a_f =\sum_{f_1 f_2 =f}\frac{\tau_{k_{1}(f_{1})}\tau_{k_{2}(f_{2})}}{|f_1|^{\alpha_{1}}|f_2|^{\alpha_{2}}}\, e^{i\big(\theta_{1}\,d(f_{1})+ \theta_{2}\,d(f_{2})\big)}.
\end{align} 
We start by defining the following truncated $L$-function which is an analog of Dirichlet polynomials over number fields: 
\begin{align*}
\mathcal{L}_{\le (k_{1}+k_{2})X}\big(\bm{v},\chi_{D}\big):=\sum_{f\in\mathcal{M}_{\leq (k_{1} + k_{2})X}} a_f \frac{\chi_{D}(f)}{|f|^{1/2}} ,
\end{align*}
where $a_f$ is defined by \eqref{definition of main coeeficient} and the parameter $X$ will be chosen later. We call $X$ as point of truncation of \eqref{product of l functions}.

\vspace{1mm}
\noindent
Using Cauchy-Schwarz inequality, we have 
\begin{align*}
\sum_{D\in \mathcal{H}_{n}}\Big|{\mathcal{L}\Big(\frac{v_{1}}{q^{1/2+\alpha_{1}}},\chi_{D}\Big)}^{k_{1}}{\mathcal{L}\Big(\frac{v_{2}}{q^{1/2+\alpha_{2}}},\chi_{D}\Big)}^{k_{2}}\,\overline{\mathcal{L}_{\le (k_{1} + k_{2})X}\big(\bm{v}, \chi_{D}\big)}\Big|\hspace*{5cm} \\
\le \left(\sum_{D\in \mathcal{H}_{n}}\Big|{\mathcal{L}\Big(\frac{v_{1}}{q^{1/2+\alpha_{1}}},\chi_{D}\Big)}^{k_{1}}{\mathcal{L}\Big(\frac{v_{2}}{q^{1/2+\alpha_{2}}},\chi_{D}\Big)}^{k_{2}}\Big|^{2}\right)^{1/2}\left(\sum_{D\in \mathcal{H}_{n}}\Big|\overline{\mathcal{L}_{\le (k_{1}+k_{2})X} \big(\bm{v},\chi_{D}\big)}\Big|^{2}\right)^{1/2}.
\end{align*}

\noindent
Therefore, we obtain
\begin{align}\label{main inequality for lower bound}
\sum_{D\in \mathcal{H}_{n}}\Big|{\mathcal{L}\Big(\frac{v_{1}}{q^{1/2+\alpha_{1}}},\chi_{D}\Big)}^{k_{1}}{\mathcal{L}\Big(\frac{v_{2}}{q^{1/2+\alpha_{2}}},\chi_{D}\Big)}^{k_{2}}\Big|^{2}\ge\; \frac{S_{1}^{2}}{S_{2}},
\end{align}
where \begin{align*}
S_{1}:=\sum_{D\in \mathcal{H}_{n}}\Big|{\mathcal{L}\Big(\frac{v_{1}}{q^{1/2+\alpha_{1}}},\chi_{D}\Big)}^{k_{1}}{\mathcal{L}\Big(\frac{v_{2}}{q^{1/2+\alpha_{2}}},\chi_{D}\Big)}^{k_{2}}\,\overline{\mathcal{L}_{\le (k_{1} + k_{2})X}\big(\bm{v},\chi_{D}\big)}\Big|
\end{align*}
and \begin{align*}
S_{2}:=\sum_{D\in \mathcal{H}_{n}}\Big|\overline{\mathcal{L}_{\le (k_{1} + k_{2})X}\big(\bm{v},\chi_{D}\big)}\Big|^{2}.
\end{align*}
Now we establish an asymptotic formula for $S_2$ and a lower bound for $S_1$.
\subsection{Estimation of the sum $S_2$}
 %Denoting $\deg(f)=d(f)$. 
 Inserting the $D$-sum after expanding square in $S_{2}$, we get
\begin{align*}
S_{2}=&\sum_{f \in \mathcal{M}_{\le (k_{1}+k_{2})X}}\sum_{f^{\prime} \in \mathcal{M}_{\le (k_{1}+k_{2})X}} \frac{a_f\,\overline{a_{f^{\prime}}}}{|ff^{\prime}|^{1/2}}\sum_{D\in \mathcal{H}_{n}}\chi_{D}(ff^{\prime}).
\end{align*}
%=&\sum_{f\in\mathcal{M}_{\leq 2(k_{1}+k_{2})X}} \frac{1}{|f|^{1/2}}\sum_{\substack{f_{1}f_{2}f_{3}f_{4}=f}}\frac{\tau_{k_{1}(f_{1})}\tau_{k_{1}(f_{2})}\tau_{k_{2}(f_{3})}\tau_{k_{2}(f_{4})}}{|f_{1}f_{2}|^{\alpha_{1}}|f_{3}f_{4}|^{\alpha_{2}}}\\
%&\times\; e^{i\theta_{1}(d(f_{1})-d(f_{2}))+ i\theta_{2}(d(f_{3})-d(f_{4}))} \sum_{D\in \mathcal{H}_{n}}\chi_{D}(f).
%\end{align*}

\noindent
\begin{case}
Assume that $ff^{\prime}\neq \square$. Observe that $a_f \ll_{\ve} |f|^{\ve}$ and using Lemma \ref{Le3}, we obtain that
\begin{align*}
S_2\ll \sqrt{|\mathcal{H}_n|}\sum_{f\in \mathcal{M}_{\le 2(k_1+k_2)X}}\frac{1}{|f|^{\frac{1}{2}-\varepsilon}}\ll \sqrt{|\mathcal{H}_n|} \, q^{2\left(\frac{1}{2}+\varepsilon\right)(k_1+k_2)X}.
\end{align*}

\noindent
Let us  choose $X=\frac{g}{2(k_1+k_2)}$ \footnote{\label{Footnote $2$}For the Theorem \ref{Th1.2}, the point of truncation will be $(k_1 +\ldots + k_m)X$ and the choice of $X$ is equal to $\frac{g}{2(k_1 +\ldots + k_m)}$.}. So, we have
\[
S_2\ll q^{\left(\frac{3}{2}+\varepsilon\right)g}.
\]
\end{case}

\begin{case}
Assume that $ff^{\prime} =\square=l^{2}$, where $l\in \mathbb{F}_{q}[t]$. By using Lemma \ref{Le2} and  $\tau_{k}(f) \ll_{\ve} |f|^{\ve}$, 
\begin{align*}
S_{2}=|\mathcal{H}_{n}|\sum_{l\in\mathcal{M}_{\leq (k_{1}+k_{2})X}} \frac{1}{|l|}\sum_{\substack{f_{1}f_{2}f_{3}f_{4}=l^{2}}}\frac{\tau_{k_{1}(f_{1})}\tau_{k_{1}(f_{2})}\tau_{k_{2}(f_{3})}\tau_{k_{2}(f_{4})}}{|f_{1}f_{2}|^{\alpha_{1}}|f_{3}f_{4}|^{\alpha_{2}}}\\
\times\; e^{i\theta_{1}(d(f_{1})-d(f_{2}))+ i\theta_{2}(d(f_{3})-d(f_{4}))} \; \prod_{P\mid l}\left(1+\frac{1}{|P|}\right)^{-1} \\
+\; \; O\left(\sum_{l\in\mathcal{M}_{\leq (k_{1}+k_{2})X}} \frac{1}{|l|}\sum_{\substack{f_{1}f_{2}f_{3}f_{4}=l^{2}}}\frac{\tau_{k_{1}(f_{1})}\tau_{k_{1}(f_{2})}\tau_{k_{2}(f_{3})}\tau_{k_{2}(f_{4})}}{|f_{1}f_{2}|^{\alpha_{1}}|f_{3}f_{4}|^{\alpha_{2}}} \right)\\
= |\mathcal{H}_{n}|\sum_{l\in\mathcal{M}_{\leq (k_{1}+k_{2})X}}\frac{b(l)}{|l|} +\;\; O\left(q^{ \ve (k_{1}+k_{2})X}\right),\hspace{2cm}
\end{align*}
where
\begin{align*}
b(l)=\sum_{\substack{f_{1}f_{2}f_{3}f_{4}=l^{2}}}\frac{\tau_{k_{1}(f_{1})}\tau_{k_{1}(f_{2})}\tau_{k_{2}(f_{3})}\tau_{k_{2}(f_{4})}}{|f_{1}f_{2}|^{\alpha_{1}}|f_{3}f_{4}|^{\alpha_{2}}}
\, e^{i\theta_{1}(d(f_{1})-d(f_{2}))+ i\theta_{2}(d(f_{3})-d(f_{4}))}  \prod_{P\mid l}\left(1+\frac{1}{|P|}\right)^{-1}.
\end{align*}

\vspace{2mm}
\noindent
We use the Perron's formula\footnote{Perron's formula in function fields comes through the Cauchy's integral formula. More precisely $\displaystyle{\sum_{f \in \mathcal{M}_{\le X}}} a_{f}=\frac{1}{2\pi i}\underset{|u|=r}{ \int}\Big(\displaystyle{\sum_{f \in \mathcal{M}}} a_{f} \,u^{\deg(f)}\Big)\frac{du}{u^{X+1}(1-u)}$, provided that the power series $\displaystyle{\sum_{f \in \mathcal{M}}} a_{f}\, u^{\deg(f)}$ is absolutely convergent in $|u|\le r<1$.} to get 
\begin{align*}
\sum_{l\in\mathcal{M}_{\leq (k_{1}+k_{2})X}}\frac{b(l)}{|l|}=\frac{1}{2\pi i}\underset{|u|=r}{\int} B(u)\frac{(qu)^{-(k_{1}+k_{2})X}}{(1-qu)}\frac{du}{u},
\end{align*}
where 
\[
B(u)=\displaystyle\sum_{l\in\mathcal{M}}b(l)u^{d(l)} 
\quad  \text{ and }  \quad r<\frac{1}{q}.
 \]
 
 \noindent
For an irreducible polynomial $P$, we observe that
\begin{align*}
&b(P)=\left(1+\frac{1}{|P|}\right)^{-1}\sum_{\substack{f_{1}f_{2}f_{3}f_{4}=P^{2}}}\frac{\tau_{k_{1}(f_{1})}\tau_{k_{1}(f_{2})}\tau_{k_{2}(f_{3})}\tau_{k_{2}(f_{4})}}{|f_{1}f_{2}|^{\alpha_{1}}|f_{3}f_{4}|^{\alpha_{2}}}\,
 e^{i\theta_{1}(d(f_{1})-d(f_{2}))+ i\theta_{2}(d(f_{3})-d(f_{4}))} \\
 &= \left(1+\frac{1}{|P|}\right)^{-1}\left(\sum_{\substack{j=1\\ \epsilon_{j}\in\{\pm 1\}}}^{2}\frac{k_{j}(k_{j}+1)}{2} \frac{e^{2i\epsilon_{j}\theta_{j}\,d(P)}}{|P|^{2\alpha_{j}}}+ \sum_{j=1}^{2}\frac{k_{j}^{2}}{|P|^{2\alpha_{j}}} + \sum_{\epsilon_{j}\in\{\pm 1\}}\frac{k_{1}k_{2}}{|P|^{\alpha_{1}+\alpha_{2}}} e^{i(\epsilon_{1}\theta_{1}+\epsilon_{2}\theta_{2})d(P)}\right),
\end{align*}
which allows us to write $B(u)$ as
\[
B(u)=\displaystyle\prod_{j=1}^{2}\mathcal{Z}^{k_{j}^{2}}(u)\displaystyle\prod_{\substack{j=1\\ \epsilon_{j}\in\{\pm 1\}}}^{2}\mathcal{Z}^{\frac{k_{j}(k_{j}+1)}{2}}\left(ue^{2i\epsilon_{j}\theta_{j}}\right)
\displaystyle\prod_{\epsilon_{j}\in\{\pm 1\}} \mathcal{Z}^{k_{1}k_{2}} \left(u e^{i(\epsilon_{1}\theta_{1}+\epsilon_{2}\theta_{2})}\right) C(u).
\]
 Here $C(u)$ is absolutely convergent for $|u|<\frac{1}{\sqrt{q}}$. Therefore, 
\begin{align}\label{main integration}
\begin{split}
&\sum_{l\in\mathcal{M}_{\leq (k_{1}+k_{2})X}}\frac{b(l)}{|l|}=\frac{1}{2\pi i}\underset{|u|=r}{\int} B(u)\frac{(qu)^{-(k_{1}+k_{2})X}}{(1-qu)}\frac{du}{u}\\
&=\frac{1}{2\pi i}\underset{|u|=r}{\int}\prod_{j=1}^{2}\mathcal{Z}^{k_{j}^{2}}(u)\prod_{\substack{j=1\\ \epsilon_{j}\in\{\pm 1\}}}^{2}\mathcal{Z}^{\frac{k_{j}(k_{j}+1)}{2}}\left(ue^{2i\epsilon_{j}\theta_{j}}\right)
\prod_{\epsilon_{j}\in\{\pm 1\}} \mathcal{Z}^{k_{1}k_{2}} \left(u e^{i(\epsilon_{1}\theta_{1}+\epsilon_{2}\theta_{2})}\right)\\
& 
\hspace{9cm} \times C(u) \frac{(qu)^{-(k_{1}+k_{2})X}}{(1-qu)}\frac{du}{u},
\end{split}
\end{align}
where $r=\frac{1}{q^{1+\ve}}$.
\subsubsection{Calculating the main term of $S_2$}
To get main term we have to shift the contour of integration \eqref{main integration} over $u$ to a circle of radius $|u|=R=\frac{1}{q^{1/2 \,+ \,\ve}}$. The  integrand has a pole at $u=\frac{1}{q}$ of order $k_{1}^{2} + k_{2}^{2} +1$ and at $u=\frac{1}{q e^{2i\epsilon_{j} \theta_{j}}}$ of order $\frac{k_{j}(k_{j}+1)}{2}$ and at $u=\frac{1}{q\,e^{i(\epsilon_{1}\theta_{1}+\epsilon_{2}\theta_{2})}}$ of order $k_{1}k_{2}$, where  $ \epsilon_{j}\in\{\pm 1\}$ and $j=1,2$.

\vspace{1mm}
\noindent
We define 
 \begin{align*}
D(u)=\prod_{j=1}^{2}\mathcal{Z}^{k_{j}^{2}}(u)\prod_{\substack{j=1\\ \epsilon_{j}\in\{\pm 1\}}}^{2} \mathcal{Z}^{\frac{k_{j}(k_{j}+1)}{2}}\left(ue^{2i\epsilon_{j}\theta_{j}}\right)\prod_{\epsilon_{j}\in\{\pm 1\}} \mathcal{Z}^{k_{1}k_{2}} \left(u e^{i(\epsilon_{1}\theta_{1}+\epsilon_{2}\theta_{2})}\right)
 C(u) \frac{(qu)^{-(k_{1}+k_{2})X}}{u(1-qu)}.
\end{align*}

\noindent 
Using the Cauchy's residue theorem\footnote{Cauchy's residue theorem says that if $\gamma$ is a simple closed, positively oriented contour in the complex plane and $f$ is analytic excepts for some points $z_1 ,\ldots, z_n $ inside $\gamma$, then $\underset{\gamma}{\oint} f(z)\,dz=2\pi i\displaystyle\sum_{k=1}^{n}\underset{z=z_k}{Res} f(z)$. }, we obtain 
\begin{align*}
\frac{1}{2\pi i}\underset{|u|=r}{\int}D(u)\,du&=\frac{1}{2\pi i}\underset{|u|=R}{\int}D(u)\,du -\underset{u=1/q}{\mathrm{Res}}D(u) -\sum_{\substack{j=1\\ \epsilon_{j}\in\{\pm 1\}}}^{2} \underset{u=1/qe^{2i\epsilon_{j}\theta_{j}}}{\mathrm{Res}}D(u)\\
&-\sum_{\epsilon_{j}\in\{\pm 1\}} \underset{u=1/qe^{i(\epsilon_{1}\theta_{1}+\epsilon_{2}\theta_{2})}}{\mathrm{Res}}D(u), \hspace{3cm}
\end{align*}
where $r=\frac{1}{q^{1+\ve}}$ and $R=\frac{1}{q^{1/2\,+\, \ve}}$.

\vspace{2mm}
\noindent
On the circle $|u|=R=\frac{1}{q^{1/2\,+\, \ve}}$, we see that the functions $\frac{1}{1-qu},$ $\frac{1}{1-que^{i(\epsilon_1\theta_1+\epsilon_2 \theta_2)}}$ and $\frac{1}{1-que^{2i\epsilon_j\theta_j}}$ are bounded. This leads 
\begin{align*}
\frac{1}{2\pi i}\underset{|u|=R}{\int}D(u)\,du\ll  q^{-(\frac{1}{2}-\ve)(k_{1}+k_{2})X}.
\end{align*}

\noindent
\subsubsection*{Evaluation of the sum of residues}
We claim that
\begin{align*}
& \underset{u=1/q}{\mathrm{Res}}D(u) +\sum_{\substack{j=1\\ \epsilon_{j}\in\{\pm 1\}}}^{2} \underset{u=1/qe^{2i\epsilon_{j}\theta_{j}}}{\mathrm{Res}}D(u)
 + \sum_{\epsilon_{j}\in\{\pm 1\}} \underset{u=1/qe^{i(\epsilon_{1}\theta_{1}+\epsilon_{2}\theta_{2})}}{\mathrm{Res}}D(u) \hspace{4cm}\\
&\hspace{2mm}\sim_{\bm{k}, \widetilde{c}} g^{k_{1}^{2} + k_{2}^{2}}\prod_{j=1}^{2}\left(\min\Big\{\frac{1}{2|\theta_{j}|}, g \Big\}\right)^{k_j(k_j +1)}\left(\min\Big\{\tfrac{1}{ | \theta_{1}-\theta_{2}|},g \Big\} \right)^{2k_{1}k_{2}} \left(\min\Big\{\tfrac{1}{ | \theta_{1}+\theta_{2}  |},g \Big\} \right)^{2k_{1}k_{2}},
\end{align*}
where
\begin{align}\label{definition of $c$}
\widetilde{c}:= \widetilde{c}_{\bm{v}}+1\; \text{with}\; \widetilde{c}_{\bm{v}}\; \text{is defined as in \eqref{contant of main theorem}}.
\end{align}

%\underset{\left\{\substack{j:\,|\theta_{j}|=O(1/g); j=1, 2\\ (1,2):|\theta_{1}\pm\theta_{2}|=O(1/g)}\right\}} {\max}\left\{\underset{g\to \infty}{\lim}\;g|\theta_{j}|,\,\underset{g\to \infty}{\lim}g|\theta_{1}\pm\theta_{2}|\right\}
\noindent
Let us define the following sets
\begin{align*}
& W_{1}=\{j\in\{1,2\}:\underset{g\to \infty}{\lim} g|\theta_{j}|< \infty\}, \quad {W}^{c}_{1}=\{j\in\{1,2\}:\underset{g\to \infty}{\lim} g|\theta_{j}|= \infty\},\\
&W_{2}=\{(1	,2):\underset{g\to \infty}{\lim} g|\theta_{1}-\theta_{2}|< \infty\}, \quad  {W}^{c}_{2}=\{(1,2):\underset{g\to \infty}{\lim}\, g|\theta_{1}-\theta_{2}|= \infty\},\\
&W_{-2}=\{(1,2):\underset{g\to \infty}{\lim} g|\theta_{1}+\theta_{2}|< \infty\}, \quad {W}^{c}_{-2}=\{(1,2):\underset{g\to \infty}{\lim} g|\theta_{1}+\theta_{2}|= \infty\}.
\end{align*}

\noindent
We call the elements of the sets $W_1$ and $W^{c}_{1}$ as finite and infinite ``single shift" respectively.  We also call the elements of the sets $W_{\epsilon 2}$ and ${W}^{c}_{\epsilon 2} ,\,  \epsilon\in \{1, -1\}$ as finite and infinite ``pair shift" respectively\footnote{\label{Footnote $5$}Note that for two dimensional correlations only one of the sets $W_{\epsilon 2}, {W}^{c}_{\epsilon 2},\, \epsilon\in \{1, -1\}$ contains the ``pair shift" $(1, 2)$ but for higher dimensional correlations either of the sets $W_{\epsilon 2}, {W}^{c}_{\epsilon 2}$ may contain more than one ``pair shift" which are of the form $(j_1, j_2).$}.
\noindent
\subsubsection*{Estimation of finite ``single shift" and ``pair shift"}
Cauchy's residue theorem allows us to write 
\begin{align*}
\underset{u=\frac{1}{q}}{\mathrm{Res}}D(u) +\sum_{\substack{j\in W_{1}\\ \epsilon_{j}\in\{\pm 1\}}} \underset{u=\frac{1}{qe^{2i\epsilon_{j}\theta_{j}}}}{\mathrm{Res}}D(u)
+\sum_{\substack{(1,2)\in W_{\eps 2}\\\eps, \epsilon_{j}\in\{\pm1\}}} \underset{u=\frac{1}{qe^{i(\epsilon_{1}\theta_{1}+\epsilon_{2}\theta_{2}})}}{\mathrm{Res}}D(u)=\underset{\Gamma}{\int}D(u)\,du,
\end{align*}
where $\Gamma$ is a circle centered at $\frac{1}{q}$ of radius $\frac{\widetilde{c}}{g}$ and $\widetilde{c}$ is defined by \eqref{definition of $c$}.
% and $$c=\underset{\left\{\substack{j:\,|\theta_{j}|=O(1/g) \\ (1,2):|\theta_{1}\pm\theta_{2}|=O(1/\delta)}\right\}} {\max}\left\{\underset{\delta\to \infty}{\lim}\;\delta|\theta_{j}|,\,\underset{\delta\to \infty}{\lim}\delta|\theta_{1}\pm\theta_{2}|\right\}\;+\;1.$$
%For sufficiently large $g$ and $s$ on $\Gamma$, .......,we have 
We apply the definition of the sets $W_1,\, W_{1}^{c}$ and $W_{\epsilon 2},\, W_{\epsilon 2}^{c}$ to write
\begin{align*}
D(u)=\left(\frac{1}{1- qu}\right)^{k_{1}^{2} + k_{2}^{2}} \prod_{\substack{j\in W_{1}\\  \epsilon_{j}\in\{\pm 1\}}}\left(\frac{1}{1-que^{2i\epsilon_{j}\theta_{j}}}\right)^{\frac{k_{j}(k_{j}+1)}{2}} \prod_{\substack{(1,2)\in W_{\eps 2}\\\epsilon_{j}\in \{\pm 1\}}} \left(\frac{1}{1-que^{i(\epsilon_{1}\theta_{1}+\epsilon_{2}\theta_{2})}}\right)^{k_1k_{2}}\\
\times \prod_{\substack{j\in {W}^{c}_{1}\\  \epsilon_{j}\in\{\pm 1\}}}\mathcal{Z}^{\frac{k_{j}(k_{j}+1)}{2}} \left(ue^{2i\epsilon_{j}\theta_{j}}\right)  \prod_{\substack{(1,2)\in{W}^{c}_{\eps 2}\\\eps_{j}\in \{\pm 1\}}}\mathcal{Z}^{k_{1}k_{2}} \left(u e^{i(\epsilon_{1}\theta_{1}+\epsilon_{2}\theta_{2})}\right) C(u) \frac{(qu)^{-(k_{1}+k_{2})X}}{u(1-qu)}\\
=\left(\frac{1}{1- qu}\right)^{k_{1}^{2} + k_{2}^{2}+1} \prod_{\substack{j\in W_{1}\\  \epsilon_{j}\in\{\pm 1\}}}\left(\frac{1}{1-que^{2i\epsilon_{j}\theta_{j}}}\right)^{\frac{k_{j}(k_{j}+1)}{2}} \prod_{\substack{(1,2)\in W_{\eps 2}\\\epsilon_{j}\in \{\pm 1\}}} \left(\frac{1}{1-que^{i(\epsilon_{1}\theta_{1}+\epsilon_{2}\theta_{2})}}\right)^{k_1k_{2}}\\
\times (qu)^{-(k_{1}+k_{2})X}\, \widetilde{E}(u),\hspace{3cm}
\end{align*}
where \begin{align*}
\widetilde{E}(u)=\prod_{\substack{j\in {W}^{c}_{1}\\  \epsilon_{j}\in\{\pm 1\}}}\mathcal{Z}^{\frac{k_{j}(k_{j}+1)}{2}} \left(ue^{2i\epsilon_{j}\theta_{j}}\right)  \prod_{\substack{(1,2)\in {W}^{c}_{\eps 2}\\\eps, \epsilon_{1},\epsilon_{2}\in \{\pm 1\}}}\mathcal{Z}^{k_{1}k_{2}} \left(u e^{i(\epsilon_{1}\theta_{1}+\epsilon_{2}\theta_{2})}\right) \frac{C(u)}{u}.
\end{align*}
Note that $\widetilde{E}(u)$ is analytic on and inside the circle $\Gamma$ and it's radius of convergence is $\gg \frac{1}{g}$. Therefore for $|u-\frac{1}{q}|=O(\frac{1}{g})$,
 \begin{align*}
\widetilde{E}(u)=\sum_{n=0}^{\infty} e_n \left(1-qu\right)^{n}.
\end{align*} 
\noindent
 Next we evaluate the integral $\displaystyle\int_{\Gamma}D(u)\,du$ which is equal to 
\begin{align}\label{residu1}
\int_{\Gamma}\frac{1}{(1-qu)^{V+1}}\left(1+\sum_{n=1}^{\infty}\, \frac{b_n}{(1-qu)^{n}} \right)\widetilde{E}(u)\, (qu)^{-(k_{1}+k_{2})X}\,du,
\end{align}
where 
\begin{align*}
&V=k_{1}^{2} + k_{2}^{2}\,+\displaystyle\sum_{j\in W_{1}}k_{j}(k_j +1) + \displaystyle \sum_{\substack{(1,2)\in W_{\eps 2}\\\epsilon\in\{\pm 1\} }}2k_{1}k_{2} \quad \text{ and }\\
&1 + \sum_{n=1}^{\infty} \frac{b_n}{(1-qu)^n}=
\prod_{\substack{j\in W_{1}\\ \epsilon_{j}\in\{\pm 1\}}} \left( 1 +\sum_{n=1}^{\infty} (-1)^{n}  \binom{\frac{k_{j}(k_{j}+1)}{2} + n} {\frac{k_{j}(k_{j}+1)}{2}} \frac{(e^{-2i\epsilon_{j}\theta_{j}} -1)^{n}}{(1-qu)^{n}} \right)\\
& \hspace{3cm}\times \prod_{\substack{(1,2)\in W_{\eps 2}\\\eps,\epsilon_{1},\epsilon_{2}\in \{\pm 1\}}} \left( 1 +\sum_{m =1}^{\infty} (-1)^{m}  \binom{k_1 k_2 + m} {k_1 k_2} \frac{(e^{-i(\epsilon_{1}\theta_{1}+\epsilon_{2}\theta_{2})} -1)^{m}}{(1-qu)^{m}} \right).
\end{align*}
\noindent
 For $n\ge 0$, we deduce that 
\begin{align*}
\int_{\Gamma}\frac{1}{(1-qu)^{V+1}}\frac{b_n}{(1-qu)^{n}}\widetilde{E}(u) (qu)^{-(k_{1}+k_{2})X}\,du\hspace{5cm}\\
=\, e_{0}b_{n}\frac{F_{V+n}((k_{1}+k_{2})X)}{(V+n)!}\,+\sum_{l=1}^{V+n}e_{l}b_{n}\frac{F_{V+n-l}((k_{1}+k_{2})X)}{(V+n-l)!},
\end{align*}
where $F_n (x)=x(x+1)(x+2)\ldots (x+n-1)$, for $n\ge 1$ and $F_{0}(x)=1$.

\vspace{1mm}
\noindent
From the choice of $X=\frac{g}{2(k_1 +k_2)}$, right hand side of the above equation becomes 
\begin{align}\label{deduction of d_n}
\frac{e_{0}\,b_{n}}{(V+n)!} \Big(\frac{g}{2}\Big)^{V+n} +  \sum_{l=1}^{V+n} \frac{e_{l}\, b_{n}\, d_{V+n-l}}{(V+n-l)!}\Big(\frac{g}{2}\Big)^{V+n-l}, 
\end{align}
where $$d_l =1+\frac{l!}{e_{V+n-l}}\left(\frac{s^{(l)}_{l-1} e_{V+n-(l+1)}}{(l+1)!} +\frac{s^{(l+1)}_{l-1} e_{V+n-(l+2)}}{(l+2)!} +\ldots
+\frac{s^{(V+n-2)}_{l-1} e_1}{(V+n-1)!} +\frac{s^{(V+n-1)}_{l-1} e_0}{(V+n)!} \right) $$ with $$ s^{(k)}_{k-i}=\sum_{1\le l_1 <\ldots<l_i\le k} l_1 \ldots l_i\, ,\;\; i=1,2,\ldots, k.$$
 The coefficients $s^{(k)}_{k-i}$ are called the Stirling numbers of first kind and $s^{(k)}_{k-i}\le (k+1)!$ (see \cite{GHP},\, equation (6.9)). For more details about $d_l$ see Appendix \ref{calculation of dn}.
 Therefore, the integral \eqref{residu1} is equal to 
\begin{equation}\label{eq19}
\begin{multlined}
 \frac{e_0 g^{V}}{2^{V}}\sum_{n=0}^{\infty}\frac{b_{n}\, g^{n}}{2^{n} (V+n)!} + \sum_{l=1}^{V}\frac{e_{l}\, g^{V-l}}{2^{V-l}}\sum_{n=0}^{\infty}\frac{b_{n}\, d_{V+n-l}\, g^{n}}{2^{n} (V+n-l)!}
\\
 + \sum_{l=1}^{\infty} e_{V+l} \sum_{n=0}^{\infty}\frac{b_{n+l}\, d_{n}\, g^{n}}{2^{n} n!}.
\end{multlined}
\end{equation}
We claim that the main contribution comes from only the first term of the above expression. To prove this, we have to find an upper bound for the coefficients $b_n$, $e_l$ and $d_n$. 

\noindent
Let us denote
\begin{align*}
M:=\max_{\substack{j\in W_{1}\\ (1,2)\in W_{\eps 2}\\\eps\in \{\pm 1\}}}\Bigg\{\frac{k_{j}(k_{j}+1)}{2}, k_1 k_2\Bigg\}, \quad  \beta:=\max_{\substack{j\in W_{1}\\ (1,2)\in W_{\eps 2}\\\eps\in \{\pm 1\}}}\Bigg\{|1-e^{2i\theta_{j}}|,|1-e^{i(\theta_{1}\pm \theta_{2})}|\Bigg\},
\end{align*}
and  $2w:=\displaystyle\max_{\substack{j}}\, \{\,|W_j|\,\}$.\,%\;%\footnote{\label{Footnote $6$}Note that for the product of $2$-shifted quadratic Dirichlet $L$-function, the value of $w$ is $1$ and for the product of $m$-shifted quadratic Dirichlet $L$-function, the value of $w$ is $m/2$.}.
We can write $b_n$ as  
 \begin{align*}
b_{n} =(-1)^{n}\sum_{\substack{\sum n_j +m_{12} =n\\ n_j,m_{12}\ge 0\\j\in W_1, (1,2)\in W_{\eps 2}\\\epsilon\in\{\pm 1\}}}\prod_{\substack{j\in W_{1}\\ \epsilon_{j}\in\{\pm 1\}}} \binom{\frac{k_{j}(k_{j}+1)}{2} + n_j} {\frac{k_{j}(k_{j}+1)}{2}} (e^{-2i\epsilon_{j}\theta_{j}} -1)^{n_j}\\
\times \prod_{\substack{(1,2)\in W_{\eps 2}\\ \epsilon_{j}\in \{\pm 1\}}}\binom{k_1 k_2 + m_{12}} {k_1 k_2} (e^{-i(\epsilon_{1}\theta_{1}+\epsilon_{2}\theta_{2})} -1)^{m_{12}}.
\end{align*}

\noindent
Note that the number of terms such that $\sum n_j +m_{12} =n$ with $ n_j,m_{12}\ge 0$ and $j\in W_1,\, (1,2)\in W_{\eps 2}$ is $\binom{w+n-1}{w-1}$. Therefore, for large $g$ and $n\ge 1$, we obtain 
\begin{align}\label{eq20}
|b_n|\le \binom{w+n-1}{w-1}{\binom{M+n}{M}}^{w} {\beta}^{n}\le a_0 n^t {\beta}^{n},
\end{align}
where $a_0$, $t$ are constants depend on $w$ and $M$.

\noindent
Let $r$ be the radius of convergence of $\widetilde{E}(u)$. Note that $\frac{1}{g} =o(r)$. Hence $\displaystyle\lim_{n\to \infty} \frac{e_{n+1}}{e_n}=\frac{1}{r}=o(g)$, and this gives
\begin{align}\label{eq21}
|e_n|\le e_0\,a_1 \,\left(\frac{2}{r}\right)^n .%\ll o\Bigg({g}^{n} 
%\prod_{j\in {W}^{c}_{1} }\frac{1}{|\theta_{j}|^{k_{j}(k_{j}+1)}}  \prod_{\substack{(1,2)\in {W}^{c}_{\eps 2}\\ \epsilon\in\{\pm 1\}}} \frac{1}{|\theta_{1}+ \epsilon\theta_{2}|^{2k_{1}k_{2}}}\Bigg),
\end{align}
where $a_1\in \mathbb{R}$ depends on $\widetilde{E}$.

\noindent
Note that, 
\begin{align*}
\frac{e_l d_{V+n-l}}{(V+n-l)!}= \frac{e_l}{(V+n-l)!} + \left(\frac{s^{(V+n-l)}_{V+n-l-1} e_{l-1}}{(V+n-l+1)!}+ \frac{s^{(V+n-l+1)}_{V+n-l-1} e_{l-2}}{(V+n-l+2)!}+\ldots+\right.\\
\left.\frac{s^{(V+n-2)}_{V+n-l-1} e_{1}}{(V+n-1)!}+ \frac{s^{(V+n-1)}_{V+n-l-1} e_{0}}{(V+n)!}\right),
\end{align*}
which implies together with \eqref{eq21}, $$\frac{e_l d_{V+n-l}}{(V+n-l)!}\le e_l + e_{l-1} +\ldots + e_1 + e_0\le e_0 a_1 \sum_{k=1}^{l}\Big(\frac{2}{r}\Big)^{k}\ll l\Big(\frac{2}{r}\Big)^{l} .$$

\noindent
By using the above bounds, the fact that $b_0 =1$ and $\frac{1}{r}=o(g)$, the second sum of \eqref{eq19} is bounded by 
\begin{align*}
&\ll_{\bm{k}} \sum_{l=1}^{V}l\left(\frac{2}{r}\right)^{l} \Big(\frac{g}{2}\Big)^{V-l}\left(1+ \sum_{n=1}^{\infty}\frac{ n^{t}(g \beta)^{n}}{2^{n}} \right)\\
&=o\left({g}^{V} 
\prod_{j\in {W}^{c}_{1}}\frac{1}{|\theta_{j}|^{k_{j}(k_{j}+1)}} \prod_{\substack{(1,2)\in{W}^{c}_{\eps 2}\\ \epsilon\in \{\pm 1\}}}\frac{1}{|\theta_{1}+\epsilon\theta_{2}|^{2k_{1}k_{2}}}\right).
\end{align*}
Since $|g \beta |< \widetilde{c}<1 $ as $g\to \infty$, the inside $n$-sum in the above expression is $O(1)$. For any $l\ge1$, using \eqref{eq20}, %and $\frac{d_n}{n!}\ll_{\bm{k}} 1$,
 we get
\begin{align*}
\sum_{n=0}^{\infty}\frac{b_{n+l} d_{n}g^{n}}{2^{n}n!}\ll_{\bm{k}}  l^{t}\beta^{l}\sum_{n=0}^{\infty}\frac{(n+l)^{t}}{l^{t}} \frac{(g \beta )^{n}}{2^{n}}\ll_{\bm{k}} l^{t}\beta^{l}\sum_{n=0}^{\infty}\frac{(n+1)^{t}( g \beta )^{n}}{2^n}\ll_{\bm{k}} l^{t}\beta^{l}.
\end{align*} 

\noindent 
From the fact $\beta/r=o(1)$ and \eqref{eq21}, the third sum of the equation \eqref{eq19} is bounded above by \begin{align*}
&\ll_{\bm{k}} \frac{1}{r^{V}}\sum_{l=1}^{\infty}\left(\frac{2\beta}{r}\right)^{l} l^{t}   \\
&= o\left(g^V \prod_{j\in {W}^{c}_{1}}\frac{1}{|\theta_{j}|^{k_{j}(k_{j}+1)}} \prod_{\substack{(1,2)\in{W}^{c}_{\eps 2}\\ \epsilon\in \{\pm 1\}}}\frac{1}{|\theta_{1}+\epsilon\theta_{2}|^{2k_{1}k_{2}}}  
\right).
\end{align*}

\noindent
Finally, we consider the first sum of the equation \eqref{eq19}. Using the bound of $b_n$ (see \eqref{eq20}), we note that
\begin{align*}
1\ll \sum_{n=0}^{\infty}\frac{b_{n}g^{n}}{2^{n}(V+n)!}\ll \frac{1}{V!}+ \sum_{n=1}^{\infty}\frac{n^t (\beta\,g)^{n}}{2^{n}(V+n)!}=O(1)
\end{align*}
where the implied constant depends on $\widetilde{c}$ and $\bm{k}  \, (i.e., V)$. Therefore, we conclude that
\begin{align*}
&\int_{\Gamma}D(u)\,du\sim_{\bm{k}, \widetilde{c}} g^V \prod_{j\in {W}^{c}_{1}}\frac{1}{|\theta_{j}|^{k_{j}(k_{j}+1)}} \prod_{\substack{(1,2)\in{W}^{c}_{\eps 2}\\ \epsilon\in \{\pm 1\}}}\frac{1}{|\theta_{1}+\epsilon\theta_{2}|^{2k_{1}k_{2}}}   \\
&\sim_{\bm{k},\widetilde{c}} g^{k_{1}^{2} + k_{2}^{2}} \prod_{j=1}^{2}\left(\min \Big\{ \frac{1}{|2\theta_{j}|},g \Big\}\right)^{k_j(k_j +1)}
\left(\min\Big\{\frac{1}{|\theta_{1}-\theta_{2}|},g\Big\}\right)^{2k_1 k_2 }\left(\min\Big\{\frac{1}{|\theta_{1}+\theta_{2}|},g\Big\}\right)^{2k_1 k_2},
\end{align*}
as required.
\subsubsection*{Evaluation of infinite ``single shift'' and ``pair shift"}
We claim that 
\begin{align}\label{infinite pairs}
\sum_{\substack{j\in W_{1}\\ \epsilon_{j}\in\{\pm 1\}}} \underset{u=\frac{1}{qe^{2i\epsilon_{j}\theta_{j}}}}{\mathrm{Res}}D(u)
+ \sum_{\substack{(1,2)\in W_{\eps 2}\\\eps, \epsilon_j\in \{\pm 1\}}} \underset{u=\frac{1}{qe^{i(\epsilon_{1}\theta_{1}+\epsilon_{2}\theta_{2}})}}{\mathrm{Res}}D(u)\hspace{5.5cm}\\
=o\left(g^{k_{1}^{2}\,+\,k_{2}^{2}} \prod_{j=1}^{2}\left(\min\Big\{\frac{1}{|2\theta_{j}|},g\Big\}\right)^{k_j(k_j +1)}
\left(\min\Big\{\frac{1}{|\theta_{1}-\theta_{2}|},g\Big\}\right)^{2k_1 k_2 }\left(\min\Big\{\frac{1}{|\theta_{1}+\theta_{2}|},g\Big\}\right)^{2k_1 k_2 }\right) \nonumber.
\end{align}
 For the sake of simplicity, we will provide all the details of the proof of the claim \eqref{infinite pairs} in the Appendix section.
 \end{case}
\subsection{Estimation of the sum $S_1$}
%Now estimate the bound for $S_1$. 
We define
\[ 
\widetilde{\mathcal{L}}(u,\chi_D):={\mathcal{L}\Big(\frac{v_{1}}{q^{1/2+\alpha_{1}}},\chi_{D}\Big)}^{k_{1}}{\mathcal{L}\Big(\frac{v_{2}}{q^{1/2+\alpha_{2}}},\chi_{D}\Big)}^{k_{2}}=\sum_{f \in \mathcal{M}} a_f \chi_{D}(f)\Big(\frac{u}{\sqrt{q}}\Big)^{d(f)},
 \] 
%i.e, in terms of $u=q^{-s}$ we rew
%\[ \mathcal{L}^{\prime}(u,\chi_D)=\sum_{f \in \mathcal{M}}\chi_{D}(f)\, a_f\, \Big(\frac{u}{\sqrt{q}}\Big)^{d(f)},
%&\]
where $a_f$ is defined by \eqref{definition of main coeeficient}. We begin with the integral 
\begin{align*}
I=\frac{1}{2\pi i}\underset{|u|=r}{\oint} \widetilde{\mathcal{L}}(u,\chi_D) \frac{u^{-(k_{1}+k_{2})X}}{(1-u)} \frac{du}{u}, \quad  r<1.
\end{align*}
\noindent
Integrating term by term to get 
\begin{align*}
I=\sum_{f \in \mathcal{M}_{\le (k_{1}+k_{2})X}} a_f \frac{\chi_{D}(f)} {|f|^{1/2}}.
\end{align*}
On the other hand we move the contour of integration to $|u|=q^{y}$, encountering a simple pole at $u=1$,  $y>\frac{1}{2}$. In doing so, we obtain
\begin{align*}
I= \widetilde{\mathcal{L}}(1,\chi_D)+ \frac{1}{2\pi i}\underset{|u|=q^{y}}{\oint} \widetilde{\mathcal{L}}(u,\chi_D) \frac{u^{-(k_{1}+k_{2})X}}{(1-u)}\frac{du}{u}.
\end{align*}

\noindent We use the Lindel\"of bound $\widetilde{\mathcal{L}}(u,\chi_D)\ll q^{\ve n}$ [\cite{AS}, Theorem $3.3$]\footnote{One can use the Theorem \ref{Th2} to get the better bound for the integral \eqref{in1}, but for our case Lindel\"of bound is enough. } to obtain
 \begin{align}\label{in1}
\frac{1}{2\pi i}\underset{|u|=q^{y}}{\oint} \widetilde{\mathcal{L}}(u,\chi_D) \frac{u^{-(k_{1}+k_{2})X}}{(1-u)}\frac{du}{u}\ll \frac{q^{\ve n}}{q^{\left((k_{1}+k_{2})X+1\right)y}}.
\end{align}
It follows that 
\begin{align}\label{appprox functional equation}
\widetilde{\mathcal{L}}(1,\chi_D)=\sum_{f \in \mathcal{M}_{\le (k_{1}+k_{2})X}} a_f \frac{\chi_{D}(f)} {|f|^{1/2}} + O_\ve\left(\frac{q^{\ve n}}{q^{\left((k_{1}+k_{2})X+1\right)y}}\right).
\end{align}

\noindent
From the approximation \eqref{appprox functional equation},
 \begin{align*}
S_1&\gg \Bigg|\sum_{D\in \mathcal{H}_{n}}{\mathcal{L}\Big(\frac{v_{1}}{q^{1/2+\alpha_{1}}},\chi_{D}\Big)}^{k_{1}}{\mathcal{L}\Big(\frac{v_{2}}{q^{1/2+\alpha_{2}}},\chi_{D}\Big)}^{k_{2}}\overline{\mathcal{L}_{\le X}\big(\bm{v},\chi_{D}\big)}\Bigg|\hspace{5.5cm}\\
&=\Bigg|\sum_{f \in \mathcal{M}_{\le (k_{1}+k_{2})X}}\sum_{f^{\prime} \in \mathcal{M}_{\le (k_{1}+k_{2})X}}\frac{a_f\,\overline{a_{f^{\prime}}}}{|ff^{\prime}|^{1/2}}\sum_{D\in \mathcal{H}_{n}}\chi_{D}(ff^{\prime})\Bigg|\, + \,O_\ve\left(q^{n\ve} |\mathcal{H}_{n}|\frac{q^{(1/2 + \,\ve)(k_{1}+k_{2})X}}{q^{\left((k_{1}+k_{2})X+1\right)y}} \right)\\
& =|S_2|+ O_\ve\left(q^{n\ve }\, |\mathcal{H}_{n}|\frac{q^{(1/2\, + \ve)(k_{1}+k_{2})X}}{q^{\left((k_{1}+k_{2})X+1\right)y}} \right).
\end{align*}
We choose $X=\frac{g}{2(k_1 + k_2)}$ and $y=\frac{2}{3}$. Hence the estimate of $S_2$ gives us
\begin{align*}
 S_1 \gg_{\bm{k}}\; |\mathcal{H}_{n}|\; g^{k_{1}^{2} + k_{2}^{2}}\, \prod_{j=1}^{2}\left(\min\Big\{\frac{1}{|2\theta_{j}|},g\Big\}\right)^{k_j(k_j +1)}
\left(\min\Big\{\frac{1}{|\theta_{1}-\theta_{2}|},g\Big\}\right)^{2k_1 k_2 }\\
\times \left(\min\Big\{\frac{1}{|\theta_{1}+\theta_{2}|},g\Big\}\right)^{2k_1 k_2 }
+ O\left(|\mathcal{H}_{n}| q^{-\frac{g}{12}\, +\,\ve n} \right).\hspace{0.5cm}
\end{align*}
Inserting the estimates of $S_1$ and $S_2$ in \eqref{main inequality for lower bound} finishes the proof of Theorem \ref{Th1}.
\section{Proof of Theorem \ref{Th2}}
\noindent
To keep things simple we use the notation $\bm{v}$ instead of  $\bm{v}^{(m)}$.
The proof of the Theorem \ref{Th2} will rely on getting an upper bound of the set  $$\Upsilon_{n}(\bm{v},V)=\#  \left\{D\in\mathcal{H}_{n}: \displaystyle\sum_{j=1}^{m}2k_j \log\Big|\mathcal{L}\Big(\frac{v_{j}}{q^{\frac{1}{2}+\alpha_{j}}},\chi_{D}\Big)\Big|\geq \mu(\bm{v},g) +V  \right\},$$
for sufficiently large $n$ and for all $V>2$, where $\mu(\bm{v},g)$ is defined by \eqref{mu}.
Recall that $$2g=n-1-\lambda,$$ where $g$ and $\lambda$ are defined by \eqref{delta} and \eqref{definition of lambda} respectively. We can write
  \begin{align}\label{th1.5e1}
     \sum_{D\in \mathcal{H}_{n}} \bigg|\mathcal{L}\Big(\frac{v_{1}}{q^{\frac{1}{2}+\alpha_{1}}},\chi_{D}\Big)\bigg|^{2k_1} \ldots \bigg|\mathcal{L}\Big(\frac{v_{m}}{q^{\frac{1}{2}+\alpha_{m}}},\chi_{D}\Big) \bigg|^{2k_m}=\displaystyle\int_{-\infty}^{\infty} \Upsilon_{n}(\bm{v},V) \exp{\big(\mu(\bm{v},g) +V\big)} dV.
  \end{align}
  
  \noindent We will estimate an upper bound of $\Upsilon_{n}(\bm{v},V)$ for different ranges of $V$. The Lemma \ref{Le8} lead us  
\begin{align*}
\sum_{j=1}^{m} 2k_j \log\Big|\mathcal{L}\Big(\frac{v_{j}}{q^{\frac{1}{2}+\alpha_{j}}},\chi_{D}\Big)\Big| \leq \frac{4g}{N+1}\sum_{j=1}^{m} k_j \log\left(\frac{1+ q^{-\alpha_{j}(N+1)}}{1+ q^{-2(N+1)}}\right)\hspace{5cm}\\ 
+ 2\Re\sum_{d(f)\leq N} \sum_{j=1}^{m} k_j \frac{a_{\alpha_{j}}\left(d(f)\right) \chi_{D}(f) \Lambda(f) {v_{j}}^{d(f)}} {|f|^{\frac{1}{2}}} + O\left(\sum_{j=1}^{m}k_j \right)\\
 \;\;\le \frac{4gK}{N+1}\log{2} + 2\Re\sum_{d(f)\leq N}\sum_{j=1}^{m}k_j \frac{a_{\alpha_{j}}\left(d(f)\right)  \chi_{D}(f)\Lambda(f){v_{j}}^{d(f)}}{|f|^{\frac{1}{2}}} + O(K),
\end{align*} 
where 
\[ K=\displaystyle\sum_{j=1}^{m}k_j\quad \text{and}\quad a_{\alpha_{j}}\left(d(f)\right)=\frac{1}{d(f)|f|^{\alpha_{j}}} -\frac{1}{d(f)|f|^{2}}+ O\left(\frac{1}{(N+1)q^{(N+1)\alpha_{j}}}\right).
\]

\noindent
Applying the prime polynomial theorem, the contribution from square polynomials $f=P^{2}$ to the second term of the right hand side of the above inequality is 
\begin{align*}
& 2\Re \sum_{d(P)\leq \frac{N}{2}}\sum_{j=1}^{m}k_j \frac{a_{\alpha_{j}}\left(2d(P)\right) \chi_{D}(P) d(P) {v_{j}}^{2d(P)}} {|P|} +O(\log\log n)\\
&\leq \mu(\bm{v},g) + \frac{2g K}{N+1}+ O(\log\log n),
\end{align*}
\noindent
where the error term $O(\log\log n)$ comes from the sum over $P$ such that $P| D$. Also it is easy to verify that the contribution from $f=P^r$ with $r\ge 3$ is $O(1)$.
Therefore, we deduce that 
\begin{align*}
\sum_{j=1}^{m}2k_j \log\Big|\mathcal{L}\Big(\frac{v_{j}}{q^{\frac{1}{2}+\alpha_{j}}},\chi_{D}\Big)\Big|\leq S_{1}(D)+S_{2}(D) +\mu(\bm{v},g)+ \frac{5gK}{N+1} +O(\log\log n),
\end{align*} 
\noindent
where 
\begin{align*}
&S_{1}(D)=2\displaystyle\sum_{d(P)\leq N_{0}}\frac{\chi_{D}(P)} {|P|^{1/2}} \displaystyle\sum_{j=1}^{m}k_j a_{\alpha_{j}}\left(d(P)\right)d(P) \cos(\theta_j d(P)) ,\\
 &S_{2}(D)=2 \displaystyle\sum_{N_{0}< d(P)\leq N} \frac{\chi_{D}(P)} {|P|^{1/2}} \displaystyle\sum_{j=1}^{m}k_j a_{\alpha_{j}}\left(d(P)\right)d(P) \cos(\theta_j d(P)). 
\end{align*}

\noindent
We rewrite $\sigma(\bm{v},g)$ as
\[
\sigma(\bm{v},g)=2\left(\sum_{j=1}^{m}k_j^2\right)\log{g} +  2\sum_{j=1}^{m}k_j^2 F_j 
+4\displaystyle\sum_{i<j}k_i k_j F_{i, j},
\] where 
\[
F_j= \log\left(\min\Big\{\frac{1}{2|\theta_{j}|}, g \Big\} \right) \, \text{ and }\, 
F_{i,j}= \log\left(\min\Big\{\tfrac{1}{ | \theta_{i}-\theta_{j}|}, g \Big\} \right)+ \log\left(\min\Big\{\tfrac{1}{ | \theta_{i}+\theta_{j}|}, g \Big\} \right).
\]

\noindent 
From now onward, for the sake of simplicity we write $\sigma(\bm{v},g)$ simply as $\bm{\sigma}$. 
We consider various different range of $V$. The range $-\infty< V\leq \sqrt{\log{g}}$ yields
\begin{align*}
\int_{-\infty}^{\infty}\Upsilon_{n}(\bm{v},V) \exp{\big(\mu(\bm{v}, g) +V\big)} dV \ll |\mathcal{H}_n| \exp{\big(\sqrt{\log{g}} + \mu(\bm{v}, g\big)}\ll|\mathcal{H}_{n}| g^{o(1)} \exp{(\mu(\bm{v},g))}. 
\end{align*}

\noindent
Applying Lemma \ref{Le1}, it is enough to assume that $\sqrt{\log{g}}\leq V\leq \frac{K g}{\log_{q} g}$. We define the quantity $A$ by 

\begin{align*}
A=\left\{
\begin{array}
[c]{ll}
\frac{\log \bm{\sigma}}{2}, &\, \text{if\,} \sqrt{\log{g}}\leq V\leq \bm{\sigma},
\vspace{1mm}
\\
\frac{\bm{\sigma}\log \bm{\sigma}}{2V}, &\, \text{if\,} \bm{\sigma}\leq V\leq \frac{\bm{\sigma}\log \bm{\sigma}}{25K},
\vspace{1mm}
\\
7K, &\, \text{if\,} V> \frac{\bm{\sigma}\log \bm{\sigma}}{25K}.
\end{array}
\right.
\end{align*}

\noindent
Let us consider 
\[
\frac{g}{N+1}=\frac{V}{A} \quad \text{ and } \quad N_{0}=\frac{N}{\log_{q}g}.
\]
\noindent
 Notice that, if $D\in \Upsilon_{n}(\bm{v},V)$ then we must have either  \[
 S_{1}(D)\geq V(1-\frac{6K}{A}):=V_{1} \quad \text{ or } \quad S_{2}(D)\geq \frac{KV}{A}:=V_{2}.
 \]
 To determine an upper bound of $\Upsilon_{n}(\bm{v},V)$, we will actually examine the set
 \[
 \Upsilon_{n}(\bm{v},V_{i})=\# \left\{ D\in \mathcal{H}_{n}: S_{i}(D) \geq V_{i} \right\},
 \]for $i=1, 2$.
 \noindent
  %Recall that 
  %\[
 % a_{\alpha_{j}}\left(d(P)\right)=\frac{1}{d(P)|P|^{\alpha_{j}}} -\frac{1}{d(P)|P|^{2}} +  O\left(\frac{1}{(N+1)q^{(N+1)\alpha_{j}}}\right).
 % \]
  \noindent
   We set $a_{j}(P):=a_{\alpha_{j}}\left(d(P)\right)  d(P) \cos(\theta_j d(P))$. So
    \[
    a_{j}(P)=\frac{\cos(\theta_j d(P))}{|P|^{\alpha_{j}}} -\frac{ \cos(\theta_j d(P))}{|P|^{2}} + O\left(\frac{d(P)}{(N+1)q^{(N+1)\alpha_{j}}}\right)\ll 1.
    \]
   
    \noindent
Using Lemma \ref{Le10}, we obtain
\begin{align*}
\sum_{D\in \mathcal{H}_{n}}|S_{2}(D)|^{2l} &\ll |\mathcal{H}_{n}| \frac{(2l)!}{l!\, 2^{l}}\Bigg(\sum_{N_{0}<d(P)\leq N}\frac{(\sum_{j=1}^{m}2k_{j})^{2}}{|P|} \Bigg)^{l}\\
&\ll |\mathcal{H}_{n}| \frac{(2l)!}{l!  2^{l}}\left(4K^{2}\left( \log\log_{q}g +  O(1)\right)  \right)^{l},
\end{align*} for any $l$ such that $2lN\leq n$, which implies that $l\leq \frac{g}{N}+\frac{1}{2N}\leq \frac{2V}{A}.$

\vspace{1mm}
\noindent
Therefore, by using Markov's inequality and Stirling's formula, it follows that
\begin{align*}
\Upsilon_{n}(\bm{v},V_{2})&\leq {V_{2}}^{-2l}\Big(\sum_{D\in \mathcal{H}_{n}}|S_{2}(D)|^{2l} \Big) \\
&\ll |\mathcal{H}_{n}| \Big(\frac{A}{KV}\Big)^{2l}\frac{(2l)!}{l! 2^{l}}\left(4 K^{2}\left( \log\log_{q}g + O(1)\right)\right)^{l}\\
&\ll |\mathcal{H}_{n}| \exp(-\frac{V}{2A}\log V).
\end{align*}
Again applying Lemma \ref{Le10} and Stirling's formula, we get 
\begin{align*}
\sum_{D\in \mathcal{H}_{n}}\big|S_{1}(D)\big|^{2l}\ll |\mathcal{H}_{n}|\frac{(2l)!}{l! 2^{l}}\Bigg(\sum_{d(P)\leq N_{0}}\frac{1}{|P|}\bigg(\sum_{j=1}^{m}\frac{2k_j \cos(\theta_j d(P))}{|P|^{\alpha_{j}}}\bigg)^2 \Bigg)^{l}\hspace{5cm}\\
\ll |\mathcal{H}_{n}| \frac{(2l)!}{l! 2^{l}}\Bigg(\sum_{d(P)\leq N_{0}}4\bigg(\sum_{j=1}^{m}\frac{{k_j}^{2} \cos^2(\theta_j d(P))}{|P|^{1+2\alpha_j}}+2\sum_{i<j}k_i k_j \frac{\cos(\theta_i d(P))\cos(\theta_j d(P))}{|P|^{1+ (\alpha_i + \alpha_j)}}\bigg) \Bigg)^{l}\\
\ll |\mathcal{H}_{n}| \left(\frac{l\bm{\sigma}}{e} \right)^{l},\hspace{12.4cm}
\end{align*} for any $l$ such that $2lN_0\le n$, which implies that $l\le \frac{V}{A}\log_{q} g$.
Markov's inequality gives us 
\begin{align*}
\Upsilon_{n}(\bm{v},V_{1})\ll {V_{1}}^{-2l}\left(\sum_{D\in \mathcal{H}_{n}}|S_{1}(D)|^{2l} \right) \ll |\mathcal{H}_{n}| \left(\frac{l\bm{\sigma}}{e {V_{1}}^{2}} \right)^{l}.
\end{align*}
\noindent
%We see that $V_{1}=V\,(1-\frac{K}{3A})$ 
It is now convenient to consider the case when $V\le \frac{\bm{\sigma}^{2}}{K^{3}}$  and the case $V> \frac{\bm{\sigma}^{2}}{K^{3}}$ separately.

\noindent
${Case\, 1.}\;$ Assume that $V\leq \frac{\bm{\sigma}^{2}}{K^{3}}$. We choose
$l=\lfloor\frac{{V_{1}}^{2}}{\bm{\sigma}}\rfloor$. The definition of $A$ and this choice of $l$ implies that $l\leq \frac{V}{A}\log_{q}g$. In this case, we find that
\begin{align*}
\Upsilon_{n}(\bm{v},V_{1})\ll |\mathcal{H}_{n}|\, \exp\left(l\log\left(\frac{l\bm{\sigma}}{e {V_{1}}^{2}} \right)  \right) \ll |\mathcal{H}_{n}| \, \exp\left(-\frac{{V_{1}}^{2}}{\bm{\sigma}} \right). 
\end{align*}
\noindent
${Case\, 2.}\;$ Assume that $V> \frac{\bm{\sigma}^{2}}{K^{3}}$. We choose $l=\lfloor10V\rfloor$. Again from the definition of $A$, it is easy to see that this choice $l$ satisfies $l\leq \frac{V}{A}\log_{q}g$.
Notice that $V> \frac{\bm{\sigma}^{2}}{K^{3}}$, implies $\log{V}>2\log{\bm{\sigma}}-3\log{K}$. So, we have
\[
A=K \quad  \text{ and } \quad  {V_{1}}^{2}=25 {V}^{2}.
\] 
Hence, we conclude that 
\begin{align*}
\Upsilon_{n}(\bm{v},V_{1})&\ll |\mathcal{H}_{n}|  \exp\left(10V\log\left(\frac{10V\bm{\sigma}}{e {V_{1}}^{2}} \right)  \right)\\
&\ll |\mathcal{H}_{n}| \exp\left(-4V\log{V} \right),
\end{align*} for sufficiently large $g$.

\noindent
Therefore combining the above estimates, we deduce that 
\begin{align}\label{Upesp1}
\Upsilon_{n}(\bm{v},V)\ll |\mathcal{H}_{n}|\left\{  \exp\left(-\tfrac{V}{2A}\log V\right) +  \exp\left(-\tfrac{{V_{1}}^{2}}{\bm{\sigma}} \right) + \exp\left(-4V\log{V}\right)\right\}.
\end{align}
\noindent
We extract the value of $V_1$ for various range of $V$ comes from the definition of $A$.

\noindent
If $\sqrt{\log{g}}\leq V\leq \bm{\sigma}$, then
\[
A=\frac{1}{2}\log \bm{\sigma}\quad\text{and}\quad V_{1}=V\Big(1-\frac{12K}{\log \bm{\sigma}}\Big).
\] 
So, for sufficiently large $g$, \eqref{Upesp1} implies that 
\begin{align*}
\Upsilon_{n}(\bm{v},V)&\ll |\mathcal{H}_{n}| \exp\left(-\tfrac{V^{2}}{\bm{\sigma}}\left(1-\tfrac{12K}{\log \bm{\sigma}} \right)^{2}  \right)\\
&\ll  |\mathcal{H}_{n}| \exp\left(-\tfrac{V^{2}}{\bm{\sigma}}\left(1-\tfrac{24K}{\log \bm{\sigma}} \right)  \right).
\end{align*}

\noindent
If  $\bm{\sigma}\leq V\leq \frac{1}{25K}\bm{\sigma}\log \bm{\sigma}$, then 
\[
A=\frac{\bm{\sigma}\log \bm{\sigma}}{2V} \quad \text{ and } \quad V_{1}=V\Big(1-\tfrac{12KV}{\bm{\sigma}\log \bm{\sigma}}\Big).
\]

\noindent
 For this range of $V$, $\frac{\log{V}}{\bm{\sigma}\log \bm{\sigma}}>\frac{1}{\bm{\sigma}}$ and hence from \eqref{Upesp1} we obtain
 \begin{align*}
\Upsilon_{n}(\bm{v},V)&\ll |\mathcal{H}_{n}|\bigg\{\exp\left(-\frac{V^{2}\log{V}} {\bm{\sigma}\log \bm{\sigma}}\right)   + \exp\left(-4V\log{V}\right) \bigg. \\
&\bigg. \hspace{3cm} + \exp\left(-\tfrac{V^{2}}{\bm{\sigma}}\left(1-\frac{12KV}{\bm{\sigma}\log \bm{\sigma}}\right)^{2}\right)\bigg\}\\
&\ll |\mathcal{H}_{n}| \exp\left(-\frac{V^{2}}{\bm{\sigma}}\left(1-\frac{24KV}{\bm{\sigma}\log \bm{\sigma}} \right)\right). 
\end{align*}

\noindent
Finally, if $V> \frac{1}{25K}\bm{\sigma}\log \bm{\sigma}$, then 
\[
A=7K \quad \text{ and } \quad V_{1}=\frac{V}{7}.
\]
So from \eqref{Upesp1}, we get that
\begin{align*}
\Upsilon_{n}(\bm{v},V)\ll |\mathcal{H}_{n}|\exp\left(-\tfrac{V}{98K}\log V \right). 
\end{align*} 
\noindent 
Adding these estimates in \eqref{Upesp1} for different range of $V$, we conclude that
\begin{align}\label{Upesp2}
\Upsilon_{n}(\bm{v},V)\ll \left\{
\begin{array}
[c]{ll}
|\mathcal{H}_{n}|n^{\ve}\,\exp\left(-\frac{V^{2}}{\bm{\sigma}} \right)  &,\, \text{if\,} 3\leq V\leq 2021\bm{\sigma},
\vspace{1mm}
\\
|\mathcal{H}_{n}|n^{\ve} \exp\left(-4V\right)  &,\, \text{if\,} V> 2021\bm{\sigma}.
\end{array}
\right.
\end{align} 
Inserting \eqref{Upesp2} in \eqref{th1.5e1} finishes the proof of Theorem \ref{Th2}.

\section{Proof of Theorem \ref{th5}} 
Let $C_{1/g}$ be the circle in the complex plane whose center is origin and radius is $\frac{1}{g}$.
By Cauchy's integral formula
\begin{equation*}
\mathcal{L}^{(l)}(q^{-1/2},\chi_{D})=\frac{l!}{2\pi i}\oint_{C_{1/g}} L\left(\frac{1}{2}+\theta,\chi_{D}\right)\,\frac{d\theta}{{\theta}^{l+1}},
\end{equation*}
Notice that if $\theta=\alpha -\frac{it}{\log q}$, then  
\[
L\Big(\frac{1}{2}+\theta,\chi_{D}\Big) %L\Big(\frac{1}{2}+\alpha-\frac{it}{\log q},\chi_{D}\Big)
=\mathcal{L}\Big(\frac{v}{q^{\alpha+\, 1/2}},\chi_{D}\Big),
\]
 where $\alpha=O\left(\frac{1}{g}\right)$. Therefore, applying H\"{o}lder's inequality, we see that  
\begin{align*}
\sum_{D\in \mathcal{H}_{n}}\big|\mathcal{L}^{(l)}(q^{-1/2},\chi_{D}) \big|^{k}&\leq \left( \frac{l!}{2\pi}\right)^{k}\Big(\sum_{D\in \mathcal{H}_{n}} \oint_{C_{1/g}} \big|L\Big(\frac{1}{2}+\theta,\chi_{D}\Big)\big|^{k}\,|d\theta|\Big) \Big(\oint_{C_{1/g}} |\theta|^{-\frac{k(l+1)}{(k-1)}}|d\theta|\Big)^{(k-1)}\\
&\ll  \left( \frac{l!}{2\pi}\right)^{k} \left( \frac{2\pi}{g}\right)^{k-1} \left( \frac{g}{ 2\pi}\right)^{k(l+1)} \Big(\sum_{D\in \mathcal{H}_{n}} \oint_{C_{1/g}} \big|L\Big(\frac{1}{2}+\theta,\chi_{D}\Big)\big|^{k}\,|d\theta|\Big)\\
&\ll \left( \frac{l!}{2\pi}\right)^{k} \left( \frac{2\pi}{g}\right)^{k} \left(\frac{g}{ 2\pi}\right)^{k(l+1)} \underset{|\theta|\le \frac{1}{g}}{\max}\sum_{D\in \mathcal{H}_{n}} \big|L\Big(\frac{1}{2}+\theta,\chi_{D}\Big)\big|^{k}.
\end{align*}

\noindent
As a direct application of Theorem \ref{Th2}, we obtain 
\begin{align*}
\sum_{D\in \mathcal{H}_{n}}\Big|\mathcal{L}\Big(\frac{v}{q^{\alpha+\, 1/2}},\chi_{D}\Big)\Big|^{k}\ll_{\ve} |\mathcal{H}_{n}|\,{g}^{\frac{k(k+1)}{2}+\ve}.
\end{align*}
Using this upper bound to the above inequality, we conclude that 
\begin{align*}
\sum_{D\in \mathcal{H}_{n}}\big|\mathcal{L}^{(l)}(q^{-1/2},\chi_{D}) \big|^{k}\ll_{\ve} |\mathcal{H}_{n}|\,{g}^{\frac{k(k+1)}{2}+kl+\ve} .
\end{align*}

\vspace*{1mm} 
\noindent \textbf{Acknowledgements:}  We thank the anonymous referees for their valuable comments and insightful suggestions that have improved the quality of the manuscript.

\section{Appendix}
\subsection{Proof of claim \eqref{infinite pairs}}
We have to show that for $j\in {W}^{c}_1$, 
\begin{align}\label{infinite diagonal pair}
\underset{u=\frac{1}{qe^{2i\theta_{j}}}}{Res} D(u)=\hspace{12cm}\\
o\left(g^{k_{1}^{2}\,+\,k_{2}^{2}}\prod_{j=1}^{2}\left(\min\Big\{\frac{1}{2|\theta_{j}|}, g \Big\}\right)^{k_j(k_j +1)}\left(\min\Big\{\tfrac{1}{ | \theta_{1}-\theta_{2}|},g \Big\} \right)^{2k_{1}k_{2}} \left(\min\Big\{\tfrac{1}{ | \theta_{1}+\theta_{2}  |},g \Big\} \right)^{2k_{1}k_{2}}\right)\nonumber , 
\end{align}
and for $(1,2)\in {W}^{c}_{2\eps}$, $\eps\in \{\pm 1\}$\footnote{\label{Footnote $8$}To estimate infinite pair shift for Theorem \ref{Th2}, one can follow the article of V. Chande [\cite{CHA}, Appendix].},
\begin{align}\label{infinite non diagonal pair}
\underset{u=\frac{1}{qe^{i(\theta_{1}+\eps \theta_2)}}}{Res} D(u)=\hspace{11cm}\\
o\left(g^{k_{1}^{2}\,+\,k_{2}^{2}}\prod_{j=1}^{2}\left(\min\Big\{\frac{1}{2|\theta_{j}|}, g \Big\}\right)^{k_j(k_j +1)}\left(\min\Big\{\tfrac{1}{ | \theta_{1}-\theta_{2}|},g \Big\} \right)^{2k_{1}k_{2}} \left(\min\Big\{\tfrac{1}{|\theta_{1}+\theta_{2}|},g \Big\} \right)^{2k_{1}k_{2}}\right)\nonumber. 
\end{align}
 We will prove the claim \eqref{infinite diagonal pair} and proof of the claim \eqref{infinite non diagonal pair} follows in the similar way. We assume that $\eps,\eps_j\in\{1,-1\}$ for $j=1,2$. To prove the claim \eqref{infinite diagonal pair}, without loss of generality, we assume that $1\in W^{c}_1$, so $(1,2)\in W_{-2}^{c}$. Note that if $2 \in {W}^{c}_1$ and $(1,2)\in {W}^{c}_{ 2}$, then they are not closed to each other i.e.,  $|\theta_{1}-\theta_{2}|\gg \frac{1}{g}$, $|\theta_{1}-(\theta_{1}-\theta_{2})|=|\theta_2|\gg \frac{1}{g}$ and $|\theta_2-(\theta_{1}+\theta_{2})|=|\theta_1|\gg\frac{1}{g}$, otherwise they will contained in the sets $W_1$ and $W_{ 2}$ respectively.
\noindent
By Cauchy's theorem, we obtain
$$\underset{u=1/qe^{2i\theta_{j}}}{Res} D(u)=\underset{\widetilde{C}}{\oint}D(u)\,du,\quad j=1,2,$$ where $\widetilde{C}$ is the circle centered at $u=1/qe^{2i\theta_{j}}$ with radius $\frac{\widetilde{c}}{g}$ and $\widetilde{c}$ is defined by \eqref{definition of $c$}. Note that, $\frac{1}{g}=o(|\theta_j|)$ and $\theta_j =o(1)$ for all $j$.
\noindent
 For $u$ on the circle $\widetilde{C}$, we write
 \[
 \mathcal{Z}(u)=(1-qu)^{-1}=(1-e^{-2i\theta_{1}})^{-1}\left(1+\frac{e^{-2i\theta_{1}}(1-que^{2i\theta_{1}})}{1-e^{-2i\theta_{1}}}\right)^{-1} .
 \]
Therefore, we get
 \[
 |\mathcal{Z}(u)|\ll \frac{1}{|\theta_1|}.
 \]
  If $2\in W_1$, then it is easy to see that $|\theta_{1}\pm \theta_2|\sim |\theta_1|$. For $u$ on the circle $\widetilde{C}$,
   \[
   \mathcal{Z}(ue^{2i\eps_j \theta_2})=(1-que^{2i\eps_j \theta_2})^{-1}=(1-e^{-2i\theta_{1}})^{-1}\left(1+\frac{e^{-2i\theta_{1}}(1-que^{2i(\eps_j \theta_2-\theta_{1})})}{1-e^{-2i\theta_{1}}}\right)^{-1},
   \]
    which implies that \[
    |\mathcal{Z}(ue^{2i\eps_j \theta_2})|\ll \frac{1}{|\theta_{1}|}.
    \]
  Also, if $(1,2)\in W_{ 2}$, then for $u$ on the circle $\widetilde{C}$, we see that
     \[
     |\mathcal{Z}(ue^{2i(\eps_1 \theta_1 + \eps_2 \theta_2)})|\ll \frac{1}{|\theta_{1}|}. 
     \] 
     
     \noindent
     For elements in the infinite single shift and pair shift, we have to partition the sets ${W}^{c}_1$, ${W}^{c}_{\eps 2}$ into three different subsets to estimate bounds for the corresponding zeta functions.
    For $2\in {W}^{c}_1$, we divide the set ${W}^{c}_1$ into three subsets. First we define
 \[
 {W}^{c}_{11}:=\left\{2\in {W}^{c}_1:\underset{g\to \infty}{\lim}\,\frac{|\theta_1|}{|\theta_{2}|}\,< \,+\infty\; \text{and}\; \underset{g\to \infty}{\lim}\,\frac{\theta_1}{\theta_{2}}\neq 1 \right\}.
 \]
  If $2\in {W}^{c}_{11}$, then for $u$ on the circle $\widetilde{C}$, $$
  |\mathcal{Z}(ue^{2i\eps_j \theta_2})|\ll \frac{1}{|\theta_2|}.$$
\noindent
Next, we consider
\[
{W}^{c}_{12}=\left\{2\in {W}^{c}_1:\underset{g\to \infty}{\lim}\,\frac{|\theta_1|}{|\theta_{2}|}=\,\infty \right\}.
\]
 For $2\in {W}^{c}_{12}$ and $u$ on the circle $\widetilde{C}$, we obtain
 \[
 |\mathcal{Z}(ue^{2i\eps_j \theta_2})|\ll \frac{1}{|\theta_1|}.
 \]
\noindent
Lastly, let
\[
{W}^{c}_{13}=\left\{2\in {W}^{c}_1:\underset{g\to \infty}{\lim}\,\frac{\theta_1}{\theta_{2}}=\,1 \right\}.
\]
 For $2\in {W}^{c}_{13}$ and $u$ on the circle $\widetilde{C}$,
 \[
 |\mathcal{Z}(ue^{2i\eps_j \theta_2})|\ll \frac{1}{|\theta_1 -\theta_2|}.
 \]
\noindent
Similarly, for $(1,2)\in {W}^{c}_{\epsilon 2}$, we define 
\[
{}^{1}{W}_{\epsilon 2}^{c}=\left\{(1,2)\in {W}^{c}_{\epsilon 2}\, : \,\underset{g\to \infty}{\lim}\,\frac{|\theta_1|}{|\theta_1 -\epsilon\theta_{2}|}\,< \,+\infty\; \text{and}\; \underset{g\to \infty}{\lim}\,\frac{\theta_1}{(\theta_1 -\epsilon\theta_{2})}\neq 1 \right\}.
\]
 In this case, for $u$ on the circle $\widetilde{C}$,
  \[
  |\mathcal{Z}(ue^{i(\theta_1 - \epsilon\theta_2)})|\ll \frac{1}{|\theta_1 -\theta_2|}.
  \]
 Let 
\[
{}^{2}{W}_{\epsilon 2}^{c}=\left\{(1,2)\in {W}^{c}_{\epsilon 2}:\underset{g\to \infty}{\lim}\,\frac{|\theta_1|}{|\theta_1 -\epsilon\theta_{2}|}\,= \,+\infty\;  \right\}.
\]
\noindent
Inside the set ${}^{2}{W}_{\epsilon 2}^{c}$, for $u$ on the circle $\widetilde{C}$,
\[
|\mathcal{Z}(ue^{i(\theta_1 -\epsilon \theta_2)})|\ll \frac{1}{|\theta_1|}.
\]
\noindent
Lastly, we consider
\[
{}^{3}{W}_{\epsilon 2}^{c}=\left\{(1,2)\in {W}^{c}_{\epsilon 2}: \underset{g\to \infty}{\lim}\,\frac{\theta_1}{(\theta_1 -\epsilon\theta_{2})}= 1 \right\}.
\]
 For $u$ on the circle $\widetilde{C}$, 
 \[
 |\mathcal{Z}(ue^{i(\theta_1 - \epsilon\theta_2)})|\ll \frac{1}{|\theta_2|}.
 \]

%For $(1,2)\in \overline{W}_3$, we define $$\overline{W}_{31}=\left\{(1,2)\in \overline{W}_3:\underset{g\to \infty}{\lim}\,\frac{|\theta_1|}{|\theta_1 +\theta_{2}|}\,< \,+\infty\; \text{and}\; \underset{g\to \infty}{\lim}\,\frac{\theta_1}{\theta_1 +\theta_{2}}\neq 1 \right\}.$$ Then $|\mathcal{Z}(ue^{i(\theta_1 + \theta_2)})|\ll \frac{1}{|\theta_1 + \theta_2|}$. 
 
%For $(1,2)\in \overline{W}_3$, we define $$\overline{W}_{32}=\left\{(1,2)\in \overline{W}_3:\underset{g\to \infty}{\lim}\,\frac{|\theta_1|}{|\theta_1 +\theta_{2}|}\,= \,+\infty\;  \right\}.$$ Then $|\mathcal{Z}(ue^{i(\theta_1 + \theta_2)})|\ll \frac{1}{|\theta_1|}$.

%For $(1,2)\in \overline{W}_3$, we define $$\overline{W}_{33}=\left\{(1,2)\in \overline{W}_3: \underset{g\to \infty}{\lim}\,\frac{\theta_1}{\theta_1 +\theta_{2}}= 1 \right\}.$$ Then $|\mathcal{Z}(ue^{i(\theta_1 + \theta_2)})|\ll \frac{1}{|\theta_2|}$. 
\noindent
 Using these bounds for the zeta functions, we conclude that
  \begin{align*}
 \underset{\widetilde{C}}{\oint}D(u)\,du\;\ll\, {g}^{\frac{k_1(k_1+1)}{2} -1}\, {|\theta_{1}|}^{-\left(k_{1}^{2}\,+\,k_{2}^{2}+\frac{k_1(k_1+1)}{2}\right)}\, \min\left\{\frac{1}{|\theta_1|},\frac{1}{|\theta_{2}|},\frac{1}{|\theta_{1}-\theta_{2}|} \right\}^{k_2 (k_2 + 1)}\\
 \times\min\left\{\frac{1}{|\theta_1|},\frac{1}{|\theta_{2}|},\frac{1}{|\theta_{1}-\theta_{2}|} \right\}^{2k_1 k_2} \min\left\{\frac{1}{|\theta_1|},\frac{1}{|\theta_{2}|},\frac{1}{|\theta_{1}+\theta_{2}|} \right\}^{2k_1 k_2}
 \end{align*}
 Using the fact $\frac{1}{|\theta_j|}=o(g)$, one can easily cheek that the integral  $\underset{\widetilde{C}}{\oint}D(u)\,du$ is equal to
 \[o \left(g^{k_{1}^{2}\,+\,k_{2}^{2}} \prod_{j=1}^{2}\left(\min\Big\{\frac{1}{|2\theta_{j}|},g\Big\}\right)^{k_j(k_j +1)}
\left(\min\Big\{\frac{1}{|\theta_{1}-\theta_{2}|},g\Big\}\right)^{2k_1 k_2 } \left(\min\Big\{\frac{1}{|\theta_{1}+\theta_{2}|},g\Big\}\right)^{2k_1 k_2 } \right), \] and we obtain the claim \eqref{infinite diagonal pair}.

\subsection{Deduction of $d_n$}\label{calculation of dn} We start with the expression \eqref{deduction of d_n}, i.e.,
\begin{align*}
e_{0}b_{n}\frac{F_{V+n}((k_{1}+k_{2})X)}{(V+n)!}\,+\sum_{l=1}^{V+n}e_{l}b_{n}\frac{F_{V+n-l}((k_{1}+k_{2})X)}{(V+n-l)!}
\end{align*}
where $F_n (x)=x(x+1)(x+2)\ldots (x+n-1)$, for $n\ge 2$ and $F_{0}(x)=1,\, F_{1}(x)=x $.
We expand $F_n(x)$ to get
\begin{align*}
F_n (x)=x\left(x^{n-1} + s^{(n-1)}_{n-2} x^{n-2} + s^{(n-1)}_{n-3} x^{n-3} +s^{(n-1)}_{n-4} x^{n-4}+\ldots +s^{(n-1)}_{0} \right)
\end{align*} with
\[
 s^{(k)}_{k-i}=\sum_{1\le l_1 <\ldots<l_i\le k} l_1 \ldots l_i\, ,\;\; i=1,2,\ldots, k.
 \]
This gives us
\begin{align*}
e_{0}b_{n}\frac{F_{V+n}(x)}{(V+n)!}\,+e_{1}b_{n}\frac{F_{V+n-1}(x)}{(V+n-1)!} +e_{2}b_{n}\frac{F_{V+n-2}(x)}{(V+n-2)!}+e_{3}b_{n}\frac{F_{V+n-3}(x)}{(V+n-3)!}\\
+\ldots +e_{V+n-1}b_{n}F_{1}(x) + e_{V+n}b_{n}F_{0}(x)\\
= \frac{e_{0}b_{n}}{(V+n)!}\left(x^{V+n} + s^{(V+n-1)}_{V+n-2} x^{V+n-1} + s^{(V+n-1)}_{V+n-3} x^{V+n-2}+\ldots +s^{(V+n-1)}_{0}x  \right)\hspace{5mm}\\
+\frac{ e_{1}b_{n}}{(V+n-1)!}\left(x^{V+n-1} + s^{(V+n-2)}_{V+n-3} x^{V+n-2} + s^{(V+n-2)}_{V+n-4} x^{V+n-3}+\ldots +s^{(V+n-2)}_{0}x  \right)\\
+ \frac{ e_{2}b_{n}}{(V+n-2	)!}\left(x^{V+n-2} + s^{(V+n-3)}_{V+n-4} x^{V+n-3} + s^{(V+n-3)}_{V+n-5} x^{V+n-4} +\ldots +s^{(V+n-3)}_{0}x  \right)\\
+\ldots+\hspace{8cm}\\
+\frac{ e_{V+n-3}b_{n}}{3!}\left(x^{3} + s^{(2)}_{1}x^{2} +s^{(2)}_{0}x \right)
+\frac{ e_{V+n-2}b_{n}}{2!}\left(x^{2} + s^{(1)}_{0}x\right)
+ e_{V+n-1}b_{n} x 
+ e_{V+n}b_{n}.
\end{align*}

\begin{align*}
:=\frac{e_{0}b_{n}}{(V+n)!}x^{V+n} +\frac{ e_{1}b_{n} d_{V+n-1} }{(V+n-1)!}x^{V+n-1} + \frac{ e_{2}b_{n} d_{V+n-2} }{(V+n-2)!}x^{V+n-2} + \frac{ e_{3}b_{n} d_{V+n-3} }{(V+n-3)!}x^{V+n-3}  \\
+\ldots+ \frac{ e_{V+n-3}b_{n}\, d_{3} }{3!}x^{3} + \frac{ e_{V+n-2}b_{n}\, d_{2} }{2!}x^{2} + e_{V+n-1}b_{n}\, d_{1}x +  e_{V+n}b_{n}\,d_0,
\end{align*}
where $d_0 =1$, and  $1\leq l \leq V+n-1$,
\[
d_l =1+\frac{l!}{e_{V+n-l}}\left(\frac{s^{(l)}_{l-1} e_{V+n-(l+1)}}{(l+1)!} +\frac{s^{(l+1)}_{l-1} e_{V+n-(l+2)}}{(l+2)!} +\ldots
+\frac{s^{(V+n-2)}_{l-1} e_1}{(V+n-1)!} +\frac{s^{(V+n-1)}_{l-1} e_0}{(V+n)!} \right).
\]

\begin{thebibliography}{111}

\bibitem{AND} J. C. Andrade,  Rudnick and Soundararajan’s theorem for function fields, {\it Finite Fields Appl.}, {\bf 37} (2016), 311-327.

	\bibitem{AK1} J. C. Andrade and J. P. Keating, The mean value of $L(\frac{1}{2},\chi)$ in the hyperelliptic ensemble, {\it J. Number Theory}, {\bf 132} (2012), 2793-2816.
	
	\bibitem{AK2} J. C. Andrade, J. P. Keating, Conjectures for the integral moments and ratios of L-functions over function fields, {\it J. Number Theory}, {\bf 142} (2014), 102-148.
	
	\bibitem{AS}  S. A. Altu\u{g} and J. Tsimerman, Metaplectic Ramanujan Conjecture Over Function Fields with Applications to Quadratic Forms, {\it Int. Math. Res. Notices}, {\bf 2014} (2013), 3465-3558.
	
	
	\bibitem{BF1}  H. M. Bui, A. Florea, Zeros of quadratic Dirichlet L-functions in the hyperelliptic ensemble, {\it Trans. Amer. Math. Soc.}, {\bf 370} (2018), 8013-8045.
	
	\bibitem{BF2} H. Bui and A. Florea, Hybrid  Euler-Hadamard  product  for  quadratic  Dirichlet  $L$-functions in function fields, {\it Proc. Lond. Math. Soc.}, {\bf 117} (2018), 65-99.
	
	\bibitem{CHA} V. Chandee,  On the correlation of shifted values of the Riemann zeta function, {\it Q. J. Math.}, {\bf 62} (2011), 545-572.
	
	
	\bibitem{CF} J. B. Conrey and D. W. Farmer, Mean values of L-functions and symmetry, {\it Int. Math. Res. Notices}, {\bf 17} (2000), 883-908.
	
	\bibitem{CFKRS} J. B. Conrey, D. W. Farmer, J. P. Keating, M. O. Rubinstein and N. C. Snaith, Integral moments of $L$-functions, {\it Proc. London Math. Soc.}, {\bf 91} (2005), 33-104.
	
   \bibitem{CFKRS2} J. B. Conrey, D. W. Farmer, J. P. Keating, M. O. Rubinstein and N. C. Snaith, Autocorrelation
    of random matrix polynomials, {\it Commun. Math. Phys.}, {\bf 237} (2003), 365-395.

	
	\bibitem{CS} J.B. Conrey, N.C. Snaith, Applications of the L-functions ratios conjectures, {\it Proc. Lond. Math. Soc.}, {\bf 94} (2007), 594-646.
	
	
	\bibitem{FL1} A. Florea, Improving the error in the mean value of $L(1/2, \chi)$ in the hyperelliptic ensemble, {\it Int. Math. Res. Notices}, {\bf 20} (2017), 6119-6148.
	
	\bibitem{FL2} A. Florea, The second and third moment of $L(1/2, \chi)$ in the hyperelliptic ensemble, {\it Forum Math.}, {\bf 29} (2017), 873-892.
	
	\bibitem{FL3} A. Florea, The fourth moment of quadratic Dirichlet L-functions over function fields, {\it Geom. Funct. Anal.}, {\bf27} (2017), 541-595.
	
	\bibitem{GS} A. Granville and K. Soundararajan, Sieving and the Erd$\ddot{o}$s-Kac theorem, {\it Equidistribution in number theory}, {\bf 237} (2007), 15-27.
	
	
	\bibitem{GHP}  R. L. Graham, D. E. Knuth and O. Patashnik,  Concrete Mathematics, {\it Addison-Wesley}, Reading, MA, 1989. 
	
	
	\bibitem{AH} A. Harper, Sharp conditional bounds for moments of the Riemann zeta function, preprint 2013, http://arxiv.org/abs/1305.4618.

   \bibitem{JUNG}	 H. Jung,  Note on the mean value of $L(\frac{1}{2}, \chi_D)$ in the hyperelliptic ensemble, {\it J. Number Theory}, {\bf 133} (2013), 2706-2714.

   \bibitem{KO} J.P. Keating, B.E. Odgers, Symmetry transitions in random matrix theory and L-functions,
   {\it Comm. Math. Phys.}, {\bf 281} (2008), 499-528.

	\bibitem{KS1} J. P. Keating and N. C. Snaith, Random Matrix Theory and $\zeta(1/2+it)$, {\it Commun. Math. Phys.},  {\bf 214} (2000), 57-89.
	
	\bibitem{KS2} J. P. Keating and N. C. Snaith, Random Matrix Theory and $L$-functions at $s=1/2$, {\it Commun. Math. Phys.},  {\bf 214} (2000), 91-110.
	
	 \bibitem{MON} H.L. Montgomery, The pair correlation of zeros of the Riemann zeta-function, {\it Proc. Symp. Pure Math.}, {\bf 24} (1973), 181-193.
	 
	 \bibitem{MI1}  M. B. Milinovich, Upper bounds for moments of $\zeta^{\prime}(\rho)$, {\it The Bulletin of the London Mathematical Society}, {\bf 42} (2010), 28-44.
	 
	 \bibitem{MI2} M.B. Milinovich, Moments of the Riemann zeta-function at its relative extrema on the critical line, {\it The Bulletin of the London Mathematical Society}, {\bf 43} (2011), 1119-1129.
	 
	\bibitem{MU} M. Munsch, Shifted moments of $L$-functions and moments of theta functions, {\it Mathematika}, {\bf 63} (2017), 196-212.
	
	\bibitem{MBT} Micah B. Milinovich and Caroline L. Turnage-Butterbaugh. Moments of products of automorphic $L$-functions, {\it J. Number Theory}, {\bf 139} (2014), 175-204.
	
	\bibitem{RAM} K Ramachandra,  Some remarks on the mean value of the Riemann zeta-function and other Dirichlet
 series-II,  {\it Hardy-Ramanujan Journal}, {\bf 3} (1980), 1-24. 
	
	 \bibitem{ROS} M. Rosen, Number theory in function fields, {\it Graduate Texts in Mathematics}, {\bf 210} (2002), Springer-Verlag, New York. 
	
	\bibitem{Rud} Z. Rudnick, Traces of high powers of the Frobenius class in the hyperelliptic ensemble, {\it Acta Arith}, {\bf 143}(1) (2010), 81--99.
	
	\bibitem{RS} Z. Rudnick and K. Soundararajan,
	Lower bounds for moments of L-functions: Symplectic and orthogonal examples, pp. 293-303 in Multiple Dirichlet Series, Automorphic Forms, and Analytic Number Theory, edited by S. Friedberg et al., Proc. Sympos.
Pure Math. {\bf 75}, Amer. Math. Soc., Providence, RI, 2006. 

	\bibitem{S1} K. Sono, Lower bounds for the moments of the derivatives of the Riemann zeta-function and Dirichlet $L$-functions, {\it  Lith. Math. J.}, {\bf 52} (2012), 420-434.
	
	\bibitem{S2} K. Sono, Upper bounds for the moments of derivatives of Dirichlet L-functions, {\it Cent. Eur. J. Math.}, {\bf 12} (2014), 848-860.
	
	 \bibitem{SOU}  K. Soundararajan, Moments of the Riemann zeta-function, {\it  Annals of Mathematics}, {\bf 170} (2009), 981-93.
	 
	 \bibitem{SY} K. Soundararajan and M. young, The second moment of quadratic twists of modular $L$-functions, {\it J. Eur. Math. Soc.}, {\bf 12} (2010), 1097-1116.	
	 
	 \bibitem{WEIL} A. Weil, Sur les courbes alg\'{e}briques et les vari\'{e}t\'{e}s qui s'en déduisent, {\it Actualit\'{e}s Sci. Ind.}, {\bf  1041 } (1948), Hermann et Cie., Paris.

\end{thebibliography}
\end{document}